\documentclass{amsart} 
\title{Permutation sorting and a game on graphs}
\date{\today}

\author[a]{C.L. Jansen$^{(1)}$} 
\author[b]{M. Scheepers$^{(2),*}$} 
\author[c]{S.L. Simon$^{(3)}$} 
\author[d]{E. Tatum$^{(4)}$} 

\address[1]{Department of Mathematics, University of Notre Dame, Notre Dame, IN 46556.
{\tt Caroline.L.Jansen.10@nd.edu} }
\address[2]{Department of Mathematics, Boise State University, Boise, ID 83725.
{\tt mscheepe@boisestate.edu} }
\address[3]{Department of Mathematical Sciences, Carnegie Mellon University, Pittsburgh, PA 15213. {\tt slsimon@andrew.cmu.edu}}
\address[4]{Department of Mathematics, Rutgers University, Piscataway, NJ 08854-8019. 
{\tt elizabeth.tatum@rutgers.edu}}

\usepackage[colorinlistoftodos]{todonotes}

\usepackage{amssymb,amsmath}
\usepackage{graphicx}
\usepackage[normalem]{ulem}
\usepackage{enumerate}
\usepackage{tikz}
\usepackage{caption}
\usepackage{subcaption}
\newtheorem{theorem}{\bf Theorem}
\newtheorem{lemma}[theorem]{\bf Lemma}
\newtheorem{corollary}[theorem]{\bf Corollary}
\newtheorem{problem}{\bf Problem}
\usepackage{MnSymbol}
\usepackage{graphicx}
\usepackage[normalem]{ulem}
\usepackage{enumerate}
\usepackage{tikz}
\usepackage{caption}
\usepackage{subcaption}
\usepackage[utf8]{inputenc}
\usepackage{amsmath}
\usepackage{amsfonts}
\usepackage{amssymb}
\usepackage{amsthm}
\usepackage{mathtools}

\usepackage{subcaption}
\usepackage{standalone}

\tikzstyle{vertex}=[circle, draw, inner sep=0pt, minimum size=10pt]
\usepackage{graphicx, scalefnt}
\usepackage{pgfplots}
\usepackage{etex}
\usetikzlibrary{decorations.pathmorphing}
\usepackage{pgflibraryarrows}
\usetikzlibrary{shapes}
\tikzstyle{decision} = [diamond, draw, text width=5em, text badly centered, node distance=3cm, inner sep=0pt]
\tikzstyle{block} = [rectangle, draw, text width=5em, text centered, rounded corners, minimum height=4em]

\newcommand*{\DashedArrow}[1][]{\mathbin{\tikz [baseline=-0.25ex,-latex, dashed,#1] \draw [#1] (0pt,0.5ex) -- (1.3em,0.5ex);}}%

    \definecolor{alertcolor}{HTML}{D20155}
    \definecolor{dkblue}{HTML}{FF7401}
    \definecolor{mdblue}{HTML}{FFBD6E}
    \definecolor{ppl}{HTML}{B202D9}
	\definecolor{ltgrn}{HTML}{C8FDEC}
	\definecolor{ltyl}{HTML}{FCFDDC}
	\definecolor{darkorange}{HTML}{A94F00}
	\definecolor{experiment}{HTML}{3A1E97}
	\definecolor{dkgrn}{HTML}{016444}


\begin{document}
\maketitle

\begin{abstract}
We introduce a game on graphs. By a theorem of Zermelo, each instance of the game on a finite graph is determined. While the general decision problem on which player has a winning strategy in a given instance of the game is unsolved, we solve the decision problem for a specific class of finite graphs. This result is then applied to a permutation sorting game to prove the optimality of a proportional bound under which TWO has a winning strategy. 
\end{abstract}

Permutation sorting by block interchanges, also called \emph{swaps}, has been studied before \cite{DC}. A constrained version of the block interchange sorting operation on permutations has been postulated to perform steps in decrypting the ciliate micronuclear genome to establish a new macronuclear genome \cite{PER}. We shall refer to this constrained block interchange operation, defined later, by the name \textbf{cds} (for \underline{c}ontext \underline{d}irected \underline{s}wap). Though any permutation is sortable by arbitrary block interchanges, not all finite permutations are sortable by \textbf{cds}. For any permutation repeated applications of \textbf{cds} produce a permutation to which \textbf{cds} no longer applies, called a \textbf{cds} fixed point. For a positive integer $n$ there are $n$ \textbf{cds} fixed points in $\textsf{S}_n$, the set of permutations of $\{1,\;2,\;\ldots,\; n\}$. From this point on we consider only permutations in $\textsf{S}_n$ for some $n$.

Two fundamental graph structures associated with a permutation encode the permutation's properties in connection with the \textbf{cds} sorting operation. One is the \emph{cycle graph} introduced in \cite{BP}, and the other is the \emph{overlap graph} that was introduced in \cite{HPR} - see also \cite{TBS}. 

 The cycle graph keeps track of the status of a permutation in terms of being sorted by block interchanges. Christie \cite{DC} used the cycle graph to compute the minimum number of \emph{arbitrary} swaps required to sort a  permutation. 
The cycle graph is also instrumental in the (linear time) characterization of \textbf{cds}-sortability of a permutation, derived in Theorem 2.5 of \cite{AHMMSW}. This characterization shows that for a \textbf{cds} sortable permutation $\pi$ only the identity permutation is a fixed point for applications of \textbf{cds} to $\pi$. 

A new phenomenon emerges when $\pi$ is not \textbf{cds} sortable: 
Several fixed points may result from applications of \textbf{cds}  to $\pi$. In \cite{AHMMSW} the \emph{strategic pile}\footnote{The strategic pile will be defined in Section 2.} of $\pi$, a structure encoded by the cycle graph of $\pi$, was discovered. The strategic pile determines the various fixed points that are achievable as results of \textbf{cds} operations.

For \textbf{cds}-non-sortable permutations variability in achievable fixed points suggests combinatorial games. The following game, played by players ONE and TWO, was introduced in \cite{AHMMSW}: A \textbf{cds} non-sortable permutation $\pi\in\textsf{S}_n$ and a subset \textsf{A} of its strategic pile \textsf{P} are given. The game $\textsf{CDS}(\pi,\textsf{A})$ 
proceeds as follows:
$\pi$ is the initial permutation. Starting with ONE  the players alternate applying  a \textbf{cds} operation to the current permutation. The game ends when a \textbf{cds} fixed point is reached. If the code for this fixed point, necessarily a member of the strategic pile \textsf{P}, is a member of \textsf{A}, then ONE wins. Else, TWO wins. Note that the length of the game is bounded by $n$. 
Zermelo's Theorem \cite{Z} implies that for each $\pi$ and \textsf{A} some player has a winning strategy in the game \textsf{CDS}($\pi$,\textsf{A}). No efficient criterion for deciding from $\pi$ and \textsf{A} which player has a winning strategy is known. In Theorem 4.4 of \cite{AHMMSW} it was shown that if $\vert\textsf{A}\vert \le \frac{1}{4}\vert \textsf{P}\vert - 2$, then TWO has a winning strategy  and if $\vert\textsf{A}\vert \ge\frac{3}{4}\vert \textsf{P}\vert$, then ONE has a winning strategy.  We use an end-game analysis in Section 5 below to give more precise bounds in terms of the value of $\vert P\vert \mod 4$.

The exhibited winning strategies simply exploited the winning player's numerical advantage only, suggesting that with a deeper analysis of available moves these two numerical bounds could be vastly improved. The overlap graph keeps track of the availability of 
\textbf{cds} moves. Understanding the effect on the corresponding overlap graph when \textbf{cds} is applied to a permutation is instrumental in such an analysis. We uncovered the graph transformation, denoted \textbf{gcds}, that captures the effect of applying a \textbf{cds} sorting operation to the corresponding permutation. To our knowledge this graph transformation is described here for the first time. 

The transformation \textbf{gcds} is applicable to any graph with a nonempty set of edges - not only to the overlap graph of a permutation. We introduce a two-player game
based on \textbf{gcds}. For finite graphs Zermelo's theorem implies that one of the players has a winning strategy. We have not solved the decision problem\footnote{In stating decision problems we shall follow established format popularized by \cite{GJ}.} of which player has a winning strategy. However, in Theorem \ref{NPGraphs} we describe a class of graphs for which it is known which player has a winning strategy. Applying this new information to overlap graphs of permutations, we prove in Theorem \ref{boundtighttheorem} that for each $n>4$ that is a multiple of $4$ there is a permutation $\pi$ of $n$ symbols such that the strategic pile $\textsf{P}$ has $n-1$ elements and yet for some set $\textsf{A}\subseteq \textsf{P}$ with $\vert \textsf{A}\vert = \frac{\vert \textsf{P}\vert -3}{4} +1 $, player ONE has a winning strategy in the game $\textsf{CDS}(\pi,\textsf{A})$. This result shows that at least in one case the bounds we obtained from an end-game analysis cannot be improved.

Our paper is organized as follows:
In Section 1 we  establish basic terminology and notation. In Section 2 we record important features of the \emph{cycle graph} of \textbf{cds} non-sortable permutations.
In Section 3 we introduce the overlap graph of a permutation. In Section 4 we introduce the new graph transformation \textbf{gcds}. In Section 5 we introduce a \textbf{gcds}-based game on finite graphs. In Section 6 we identify a class of graphs for which we know which player has a winning strategy. The findings from Section 6 are used in Section 7 to prove Theorem \ref{boundtighttheorem}. 

\section{Notation and Terminology}

For $n$ a positive integer the permutations in the symmetric group $\textsf{S}_n$ are the one-to-one functions from the set $\{1,\; 2,\; \cdots,\; n\}$ to itself.  For permutation $\alpha$ the notation $\alpha = \lbrack a_1,\; \cdots,\; a_n\rbrack$ denotes that $\alpha$ maps $a_i$ to $i$. To define the \emph{context directed swap} sorting operation, \textbf{cds}, we define the notion of a pointer: Consider the entry $a_i$ of  
$\alpha$ and suppose $a_i$ is the positive integer $k$. Then the \emph{head pointer}, or simply \emph{head} of $a_i$ is the ordered pair $(k,\;  k +1)$, while the \emph{tail pointer}, or simply \emph{tail}, of $a_i$ is the ordered pair $( k-1, k)$.
We use the notation $\alpha = [a_1, a_2, \ldots ,_{(k-1\;,k)}{a_i}_{(k,\;k+1)},\ldots ,{a_j},\ldots, a_n]$ to emphasize pointer locations.

For pointers $p = (x,\; x+1)$ and $q = (y,\; y+1)$ of $\alpha$ we say that $p$ and $q$ \emph{interlock} when $p$ and $q$ appear in the order $p \ldots q \ldots p \ldots q$, or  in the order $q \ldots p \ldots q \ldots p$. We leave $\textbf{cds}_{p,q}(\alpha)$ undefined when $p$ and $q$ do not interlock. 
If $p$ and $q$ interlock we define the permutation $\beta= {\textbf{cds}}_{p,\; q}(\alpha)$ as follows:
{\flushleft{\underline{Case 1:}}} If
$\alpha = \lbrack \ldots,\; a,\; \underbrace{{ \color{black!100} {{\bf _px+1},\;{\bf  b},\;\ldots,{\bf c},\;\; {\bf y_q}}}}_{I},\;d,\;\ldots,\; e,\; x_p,\; \underbrace{{\color{black!100}{\bf f},\; \ldots,\; {\bf g}}}_{II},\;_qy+1,\;h,\; \ldots \rbrack$, define
\[
  \beta = {\textbf{cds}}_{p,\; q}(\alpha) = \lbrack \ldots,\; a,\;\underbrace{{\color{black!100}{\bf f},\; \ldots,\; {\bf g}}}_{II},\;d,\;\ldots,\; e,\; x_p,\; \underbrace{{\color{black!100}_{\bf p}{\bf x+1},\; {\bf b},\;\ldots,{\bf c},\;\; {\bf y_q}}}_{I} ,\;_qy+1,\;h,\;\ldots \rbrack
\]
{\flushleft{\underline{Case 2:}}} If
$\alpha = \lbrack \ldots,\; a,\;x_p,\; \underbrace{{\color{black!100} {\bf b},\;\ldots,{\bf c},\;\; {\bf y_q}}}_{I},\;d,\;\ldots,\;e,\;  \underbrace{{\color{black!100}{\bf  _px+1},\;{\bf  f},\; \ldots,\; {\bf g}}}_{II},\;_qy+1,\;h,\;\ldots \rbrack$, define
\[
  \beta = {\textbf{cds}}_{p,\; q}(\alpha) = \lbrack \ldots,\; a,\; x_p,\; \underbrace{{\color{black!100}{\bf _px+1},\; {\bf f},\; \ldots,\; {\bf g}}}_{II},\;d,\;\ldots,\;e,\;  \underbrace{{\color{black!100} {\bf b},\;\ldots,{\bf c},\;\; {\bf y_q}}}_{I} ,\;_qy+1,\;h,\;\ldots \rbrack
\]
{\flushleft{\underline{Case 3:}}} If
$\alpha = \lbrack \cdots,\; a,\;x_p,\; \underbrace{{\color{black!100}{\bf b}\;\ldots,{\bf c}}}_{I},\; _qy+1,\;d,\;\ldots,\; e,\; \underbrace{{\color{black!100} {\bf _px+1},\; {\bf f},\; \ldots,\; {\bf g},\;{\bf y_q}}}_{II},\;h,\;\ldots \rbrack$, define
\[
  \beta = {\textbf{cds}}_{p,\; q}(\alpha) = \lbrack \ldots,\; a,\; x_p,\; \underbrace{{\color{black!100}{\bf _px+1},\; {\bf f},\; \ldots,\; {\bf g},\; {\bf y_q}}}_{II},\;_qy+1,\;d,\;\ldots,\; e,\; \underbrace{{\color{black!100} {\bf b},\;\ldots,\;{\bf c}}}_{I},\;h,\;\ldots \rbrack
\]
{\flushleft{\underline{Case 4:}}} If
$\alpha = \lbrack \cdots,\; a,\; \underbrace{{\color{black!100}{\bf _px+1},\; {\bf b},\;\ldots,\; {\bf c}}}_{I},\; _qy+1,\;d,\;\ldots,\;e,\;x_p,\; \underbrace{{\color{black!100} {\bf f},\; \ldots,\; {\bf g},\;{\bf y_q}}}_{II},\;h,\;\ldots \rbrack$, define
\[
  \beta = {\textbf{cds}}_{p,\; q}(\alpha) = \lbrack \ldots,\; a,\; \underbrace{{\color{black!100}  {\bf f},\; \ldots,\; {\bf g},\; {\bf y_q}}}_{II},\;_qy+1,\;d,\;\ldots,\;e,\;x_p,\; \underbrace{{\color{black!100} {\bf _px+1},\; {\bf b},\;\ldots,\; {\bf c}}}_{I},\;h,\;\ldots \rbrack
\]
Pointers $p$ and $q$ no longer interlock in $\beta$. The symmetric group $\textsf{S}_n$ has $n$ elements to which \textbf{cds} does not apply, namely the identity permutation and for each $k<n$, the permutation $\lbrack k+1,\ldots,\; n-1,\; n,\; 1,\;2,\; \ldots,\; k-1,\; k\rbrack$. These are the \textbf{cds} fixed points in $\textsf{S}_n$.

As an example we perform ${\textbf{cds}}_{(3,4),(6,7)}$ on $\beta = [3,6,5,2,4,8,1,7]$: Label
$\beta$ with the relevant pointers:
$\beta = [3_{(3,4)},6_{(6,7)},5,2,{}_{(3,4)}4,8,1,_{(6,7)}7]$.
 Performing the swap yields $\delta = [3,{4,8,1},5,2,{6},7]$. 
Both $3$ and $4$ and $6$ and $7$ are now adjacent in $\delta$ and pointers $(3,4)$ and $(6,7)$ no longer interlock.

\section{Cycle graphs and the strategic pile}

\paragraph{}  The \textit{cycle graph} of a permutation $\alpha \in S_n$ is constructed as follows:  
The vertices of the graph are the numbers 0 through $n+1$.  Directed dotted edges are drawn from $i$ to $i+1$ for $i$ from 0 to $n$.
Directed black edges are drawn from each entry in $\alpha$ to the entry immediately preceding it.  Directed black edges are also
drawn from $n+1$ to the last permutation entry, and from the first permutation entry to $0$.
The cycle graph of the example $\beta = [3,6,5,2,4,8,1,7]$ is given in Figure \ref{fig:CycleGraphSample}.
\begin{figure}[h]
\centering
\begin{tikzpicture}[scale = .2]
\coordinate (v0) at (3,-10);
\coordinate (v1) at (-3,-10);
\coordinate (v2) at (-8, -6);
\coordinate (v3) at (-10, 0);
\coordinate (v4) at (-8, 6);
\coordinate (v5) at (-3, 10);
\coordinate (v6) at (3, 10);
\coordinate (v7) at (8,6);
\coordinate (v8) at (10,0);
\coordinate (v9) at (8,-6);

\node at (v0) [below] {0};
\node at (v1) [below] {1};
\node at (v2) [left] {2};
\node at (v3) [left] {3};
\node at (v4) [left] {4};
\node at (v5) [above] {5};
\node at (v6) [above] {6};
\node at (v7) [right] {7};
\node at (v8) [right] {8};
\node at (v9) [right] {9};

\draw [very thick, dotted, color = black, ->] (v0) -- (v1);
\draw [very thick, dotted, color = black, <-] (v2) -- (v1);
\draw [very thick, dotted, color = black, ->] (v2) -- (v3);
\draw [very thick, dotted, color = black, <-] (v4) -- (v3);
\draw [very thick, dotted, color = black, ->] (v4) -- (v5);
\draw [very thick, dotted, color = black, <-] (v6) -- (v5);
\draw [very thick, dotted, color = black, ->] (v6) -- (v7);
\draw [very thick, dotted, color = black, <-] (v8) -- (v7);
\draw [very thick, dotted, color = black, ->] (v8) -- (v9);

\draw [thick, color = black, <-] (v0) -- (v3);
\draw [thick, color = black, ->] (v6) -- (v3);
\draw [thick, color = black, <-] (v6) to [out=-165,in=-15, looseness=1] (v5);
\draw [thick, color = black, <-] (v5) -- (v2);
\draw [thick, color = black, ->] (v4) -- (v2);
\draw [thick, color = black, <-] (v4) -- (v8);
\draw [thick, color = black, <-] (v8) -- (v1);
\draw [thick, color = black, <-] (v1) -- (v7);
\draw [thick, color = black, <-] (v7) -- (v9);

\end{tikzpicture}
\caption{The cycle graph of $\beta = [3,6,5,2,4,8,1,7]$}\label{fig:CycleGraphSample}
\end{figure}
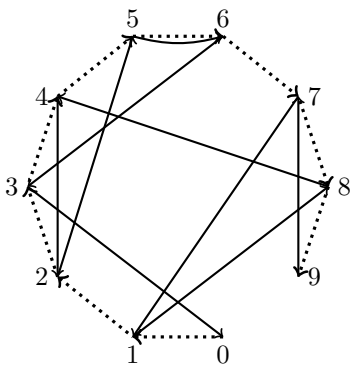 

A directed path in the cycle graph is said to be \emph{alternating} if the edges in consecutive segments of the path alternate between being dotted or black. The set of edges of the cycle graph has a unique decomposition into edge disjoint alternating cycles. Consider the 
edges $0{\color{black}\DashedArrow[ dotted, thick    ] } 1$ and $n{\color{black} \DashedArrow[ dotted, thick    ] } n+1$:  The \emph{strategic pile} of the permutation $\alpha$ is defined to be the set of pointers
\[
   \{(i,\; i+1):\; i {\color{black} \DashedArrow[ dotted, thick    ] } i+1 \mbox{ is an edge in the directed alternating path from } n{\color{black} \DashedArrow[ dotted, thick    ] } n+1 \mbox{ to } 0{\color{black}\DashedArrow[ dotted, thick    ] } 1\}.
\]
When these two edges belong to different cycles, then the strategic pile of $\alpha$  is empty. The notion of the strategic pile of a permutation was introduced and studied in \cite{AHMMSW}.

The cycle graph of  the example $\beta$ has only one alternating cycle, shown in Figure \ref{fig:AlternatingCycles}.
\begin{figure}[h]
\begin{tikzpicture}[scale = .25]
\coordinate (r0) at (2,-5);
\coordinate (b1) at (-2,-5);
\coordinate (r8) at (-5,-4.5);
\coordinate (b9) at (-8,-3);
\coordinate (r7) at (-9, -1.5);
\coordinate (b8) at (-9.5,0);
\coordinate (r4) at (-9,1.5);
\coordinate (b5) at (-8,3);
\coordinate (r6) at (-5,4.5);
\coordinate (b7) at (-2,5);
\coordinate (r1) at (2,5);
\coordinate (b2) at (5,4.5);
\coordinate (r5) at (8,3);
\coordinate (b6) at (9,1.5);
\coordinate (r3) at (9.5,0);
\coordinate (b4) at (9,-1.5);
\coordinate (r2) at (8,-3);
\coordinate (b3) at (5,-4.5);

\node at (r0) [below] {0};
\node at (r1) [above] {1};
\node at (r2) [right] {2};
\node at (r3) [right] {3};
\node at (r4) [left] {4};
\node at (r5) [right] {5};
\node at (r6) [above] {6};
\node at (r7) [left] {7};
\node at (r8) [below] {8};
\node at (b1) [below] {1};
\node at (b2) [above] {2}; 
\node at (b3) [below] {3};
\node at (b4) [right] {4};
\node at (b5) [left] {5};
\node at (b6) [right] {6};
\node at (b7) [above] {7};
\node at (b8) [left] {8};
\node at (b9) [left] {9};

\draw [very thick, dotted, color = black, ->] (r0) -- (b1);
\draw [very thick, dotted, color = black, <-] (b2) -- (r1);
\draw [very thick, dotted, color = black, ->] (r2) -- (b3);
\draw [very thick, dotted, color = black, <-] (b4) -- (r3);
\draw [very thick, dotted, color = black, ->] (r4) -- (b5);
\draw [very thick, dotted, color = black, <-] (b6) -- (r5);
\draw [very thick, dotted, color = black, ->] (r6) -- (b7);
\draw [very thick, dotted, color = black, <-] (b8) -- (r7);
\draw [very thick, dotted, color = black, ->] (r8) -- (b9);

\draw [thick, color =  black, <-] (r0) -- (b3);
\draw [thick, color =  black, ->] (b6) -- (r3);
\draw [thick, color =  black, <-] (r6) -- (b5);
\draw [thick, color =  black, <-] (r5) -- (b2);
\draw [thick, color =  black, ->] (b4) -- (r2);
\draw [thick, color =  black, <-] (r4) -- (b8);
\draw [thick, color =  black, <-] (r8) -- (b1);
\draw [thick, color =  black, <-] (r1) -- (b7);
\draw [thick, color =  black, <-] (r7) -- (b9);

\end{tikzpicture}
\caption{The alternating cycles of $\beta$}\label{fig:AlternatingCycles}
\end{figure}
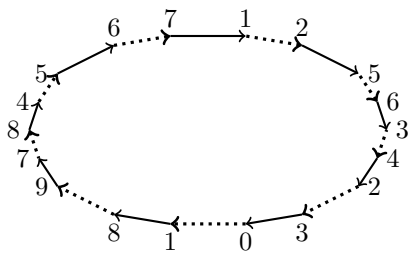

The strategic pile of $\beta$ is nonempty: As the reader may verify, the strategic pile of $\beta$ consists of all the pointers in $\beta$.
\begin{lemma}\label{strpilelemma} Let $\alpha$ be a permutation in $\textsf{S}_n$.
\begin{enumerate}
\item{(\textbf{cds} Sortability Theorem, \cite{AHMMSW}, Theorem 2.5) $\alpha$ is \textbf{cds}-sortable if, and only if, the strategic pile of $\alpha$ is empty.}
\item{(Strategic Pile Removal Theorem, \cite{AHMMSW}, Theorem 2.15 ) Suppose that $\alpha$ has more than one element in the strategic pile. For each pointer $p$ in the strategic pile there exists a pointer $q$ such that \textbf{cds}  applied to the pointer pair $p$, $q$ results in removal of $p$ from the strategic pile.}
\item{(\textbf{cds} Fixed Point Theorem, \cite{AHMMSW}, Theorem 2.19) For an $i\in\{2,\cdots,n\}$,  $[i, i+1, \ldots, n, 1, 2, \ldots, i-1]$ is an achievable \textbf{cds} fixed point of $\alpha$ if, and only if, $(i-1,i)$ is a member of the strategic pile of $\alpha$.}
\item{(\textbf{cds} Bounded Removal Theorem, \cite{AHMMSW}, Theorem 2.21)  An application of \textbf{cds} removes at most two elements from the strategic pile.}
\end{enumerate}
\end{lemma}

\section{Overlap graphs} 

Consider a permutation $\alpha\in\textsf{S}_n$. For a pointer $p = (i,\; i+1)$ in $\alpha$, draw an arc from $p$ to $p$ as in Figure \ref{fig:arcs}:
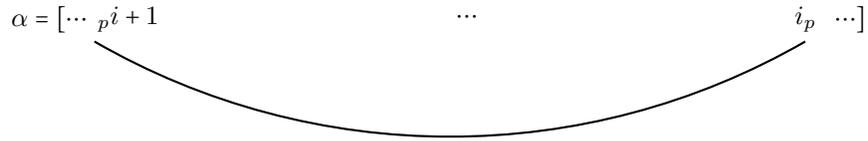
\begin{figure}[ht]
\begin{tikzpicture}[scale = 1.5]
\coordinate (v1) at (-3,3);
\coordinate (v2) at (3,3);
\coordinate (v3) at (0,3);
\coordinate (v4) at (-3.7,3);
\coordinate (v5) at (3.4,3);

    \node[circle, inner sep=8pt] at (v1){$_pi+1$};
    \node[circle, inner sep=8pt] at (v2){$i_p$};
    \node[circle, inner sep=8pt] at (v3){$\cdots$};
    \node[circle, inner sep=8pt] at (v4){$\alpha = \lbrack \cdots$};
    \node[circle, inner sep=8pt] at (v5){$\cdots\rbrack$};

    \draw [thick] (3.0,2.8) arc (-60:-120:6.3cm); 
\end{tikzpicture}
\caption{An arc from pointer $p$ to pointer $p$ in $\alpha$}\label{fig:arcs}
\end{figure}

Do this for each pointer. The set of vertices $V$, of the \emph{overlap graph} of $\alpha$, $\mathcal{O}(\alpha)$, is the set of pointers of $\alpha$. For two vertices $p$ and $q$, the overlap graph has an edge between $p$ and $q$ if the arc associated with $p$ and the arc associated with $q$ have nonempty intersection.

As an illustration consider the permutation $\beta = \lbrack 3,\; 1,\; 4,\; 2,\; 5\rbrack$. Drawing the appropriate arcs associated with the four pointers of $\beta$ produces Figure \ref{fig:OverlapGraphEx}.

\begin{figure}[ht]
\begin{tikzpicture}[scale = 3]

\coordinate (v0) at (-3,0);
\coordinate (v1) at (-2,0);
\coordinate (v2) at (-1,0);
\coordinate (v3) at (0,0);
\coordinate (v4) at (1,0);
\coordinate (v5) at (2,0);
\coordinate (v6) at (3,0);
\coordinate (v7) at (4,0);

\coordinate (t0) at (-3.1,-0.05);
\coordinate (h0) at (-2.9,-0.05);
\coordinate (h1) at (-1.9,-0.05);
\coordinate (t2) at (-1.1,-0.05);
\coordinate (h2) at (-.9,-0.05);
\coordinate (t3) at (-.1,-0.05);
\coordinate (h3) at (.1,-0.05);
\coordinate (t4) at (.9,-0.05);

\node at (v0) [below] {3};
\node at (v1) [below] {1};
\node at (v2) [below] {4};
\node at (v3) [below] {2};
\node at (v4) [below] {5};

\node at (t0) [below] {t};
\node at (h0) [below] {h};
\node at (h1) [below] {h};
\node at (t2) [below] {t};
\node at (h2) [below] {h};
\node at (t3) [below] {t};
\node at (h3) [below] {h};
\node at (t4) [below] {t};

[3,\;1,\;4,\;2,\;5]
\draw [thick, densely dotted, color=black] (h3)+(-0.0,-.15) arc (-60:-120:3.2cm); 
\draw [thick,  loosely dotted, color=black] (t2)+(0.0,-.15) arc (-60:-120:1.75cm); 
\draw [thick,  dotted, color=black] (t3)+(0,-.15) arc (-60:-120:1.7cm); 
\draw [thick,  dashed, color=black] (t4)+(0,-.15) arc (-60:-120:1.75cm); 

\end{tikzpicture}
\caption{Arcs between pointers of $\beta$. ``t" and ``h" denote the tail- and head- pointers. }\label{fig:OverlapGraphEx}
\end{figure}
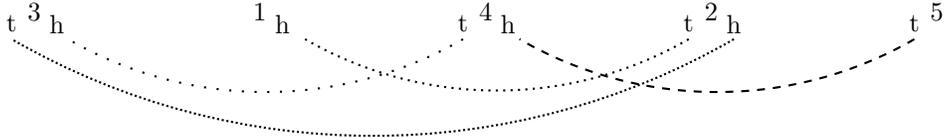
The corresponding overlap graph, including edges, is given in Figure \ref{fig:MoveGrapEx}

\begin{figure}[h]
\centering
\begin{tikzpicture}[scale = 1.0]
\coordinate (v1) at (3,0); 
\coordinate (v2) at (2,0); 
\coordinate (v3) at (1,0); 
\coordinate (v4) at (0,0); 

    \node[draw,circle, inner sep=2pt] at (v1){};
    \node[draw,circle, inner sep=2pt] at (v2){};
    \node[draw,circle, inner sep=2pt] at (v3){};
    \node[draw,circle, inner sep=2pt] at (v4){};


    \node at (v1) [below right]{(3,4)};
    \node at (v2) [below]{(2,3)};
    \node at (v3) [below]{(4,5)};
    \node at (v4) [below left]{(1,2)};

\draw [thick] (v4) -- (v1); 
\draw [thick] (v3) -- (v2);  
\draw [thick] (v4) -- (v3); 

\end{tikzpicture}
\caption{The overlap graph of $\beta = [3,1,4,2,5]$.}\label{fig:MoveGrapEx}
\end{figure}
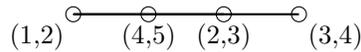

An edge between two vertices of the overlap graph denotes that the pair of pointers connected by the edge interlock. We also refer to the overlap graph of a permutation as its \textbf{cds}-\emph{move graph}.

\section{The graph transformation $\textbf{gcds}$}

We now define the graph transformation \textbf{gcds}. The reader could consult the examples in Figures \ref{fig:gcdsmovegraph} and \ref{fig:gcdsmovegraphex2} as illustrations of the graph transformation \textbf{gcds}. 
Let a graph $\mathcal{G} = (V,\; E)$ be given: $V$ is its set of vertices and $E$ is its set of edges.  Let $\{x,\; y\}\in E$ be given. Then the graph $\mathcal{G}^{\prime} = (V,E^{\prime}) = \textbf{gcds}(\mathcal{G},\{x,\; y\})$ is defined as follows: \\

\fbox{
\parbox{\textwidth}{

{\flushleft{(1)}} In $\textbf{gcds}(\mathcal{G},\{x,\; y\})$ both $x$ and $y$ are isolated vertices.\\ 
{\flushleft{(2)}}  
Create a two-columned \textit{master list} $\textsf{M}(x,y)$: 
\begin{quote}{\tt The first column contains all vertices, except y, that in $\mathcal{G}$ are neighbors\stepcounter{footnote}\footnotemark[\thefootnote] of $x$, and the second column contains all vertices, except x, that in $\mathcal{G}$ are neighbors of $y$.} 
\end{quote} 

{\flushleft{}} 
For any pair of vertices $p\neq q$ in $V\setminus\{x,\; y\}$,
determine whether $\{p,\; q\}\in E^{\prime}$ by using the following criterion:\\ \begin{itemize}
 \item[Case a:]{ If either $p$ or $q$ is not in the master list $\textsf{M}(x,\;y)$, then $\{p,\; q\}\in E^{\prime}$ if, and only if, $\{p,\; q\}\in E$.}
\item[Case b:]{ If both $p$ and $q$ are in the master list, check if both appear in the same column of the master list. 
\begin{itemize}
  \item[Subcase b.1] If not, then $\{p,\;q\}\in E^{\prime}$ if, and only if, $\{p,\; q\}\not\in E$.   \item[Subcase b.2] If they are in the same column of the master list and the number of times $p$ and $q$ occur in the master list is even, then $\{p,\; q\}\in E^{\prime}$ if, and only if, $\{p,\; q\}\in E$.  
  \item[Subcase b.3] If they are in the same column of the master list and the number of times $p$ and $q$ occur in the master list is odd, then $\{p,\; q\}\in E^{\prime}$ if, and only if, $\{p,\; q\}\not\in E$.
\end{itemize}
}
\end{itemize}
}}\footnotetext[\thefootnote]{Two vertices a and b of a graph are neighbors if $\{a,\; b\}$ is an edge of the graph.}
\vspace{0.1in}

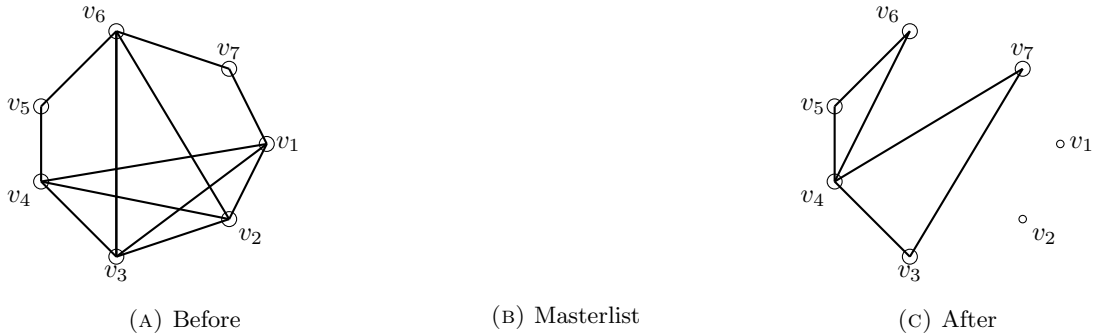
\begin{figure}[ht]
\centering
\begin{subfigure}{0.3\textwidth}
\begin{tikzpicture}[scale = .5]

\coordinate (v1) at (3,0);
\coordinate (v2) at (2,-2);
\coordinate (v3) at (-1,-3);
\coordinate (v4) at (-3,-1);
\coordinate (v5) at (-3,1);
\coordinate (v6) at (-1,3);
\coordinate (v7) at (2,2);

    \node[draw,circle, inner sep=2pt] at (v1){};
    \node[draw,circle, inner sep=2pt] at (v2){};
    \node[draw,circle, inner sep=2pt] at (v3){};
    \node[draw,circle, inner sep=2pt] at (v4){};
    \node[draw,circle, inner sep=2pt] at (v5){};
    \node[draw,circle, inner sep=2pt] at (v6){};
    \node[draw,circle, inner sep=2pt] at (v7){};
    \node at (v1) [right]{$v_1$}; 
    \node at (v2) [below right]{$v_2$}; 
    \node at (v3) [below]{$v_3$}; 
    \node at (v4) [below left]{$v_4$}; 
    \node at (v5) [left]{$v_5$}; 
    \node at (v6) [above left]{$v_6$}; 
    \node at (v7) [above ]{$v_7$}; 
    
    \draw [black, thick] (v1) -- (v2);
    \draw [black, thick] (v2) -- (v3);
    \draw [black, thick] (v3) -- (v4);
    \draw [black, thick] (v4) -- (v5);
    \draw [black, thick] (v5) -- (v6);
    \draw [black, thick] (v6) -- (v7);
    \draw [black, thick] (v3) -- (v1);
    \draw [black, thick] (v4) -- (v2);
    \draw [black, thick] (v6) -- (v3);
    \draw [black, thick] (v4) -- (v1);
    \draw [black, thick] (v6) -- (v2);
    \draw [black, thick] (v6) -- (v3);
    \draw [black, thick] (v7) -- (v1);
\end{tikzpicture}
\caption{Before} 
\end{subfigure}
\begin{subfigure}{0.3\textwidth}
\centering
\vspace{0.7in}
\begin{tikzpicture}[scale = .5]
\begin{tabular}{|l|l|} \hline
${\mathbf v_1}$ & ${\mathbf v_2}$ \\ \hline
$v_3$ & $v_6$ \\ \hline
$v_4$ & $v_4$ \\ \hline
$v_7$ & $v_3$ \\ \hline
\end{tabular}
\end{tikzpicture}
\vspace{0.8in}
\caption{Masterlist}\label{fig:Masterlist}
\end{subfigure}
\begin{subfigure}{0.3\textwidth}
\centering
\begin{tikzpicture}[scale = .5]

\coordinate (v1) at (3,0);
\coordinate (v2) at (2,-2);
\coordinate (v3) at (-1,-3);
\coordinate (v4) at (-3,-1);
\coordinate (v5) at (-3,1);
\coordinate (v6) at (-1,3);
\coordinate (v7) at (2,2);

    \node[draw, circle, inner sep=1pt] at (v1) {}; 
    \node[draw, circle, inner sep=1pt] at (v2) {}; 
    \node[draw,circle, inner sep=2pt] at (v3){};
    \node[draw,circle, inner sep=2pt] at (v4){};
    \node[draw,circle, inner sep=2pt] at (v5){};
    \node[draw,circle, inner sep=2pt] at (v6){};
    \node[draw,circle, inner sep=2pt] at (v7){};
    \node at (v1) [right]{$v_1$}; 
    \node at (v2) [below right]{$v_2$}; 
    \node at (v3) [below]{$v_3$}; 
    \node at (v4) [below left]{$v_4$}; 
    \node at (v5) [left]{$v_5$}; 
    \node at (v6) [above left]{$v_6$}; 
    \node at (v7) [above ]{$v_7$}; 
    
    \draw [black, thick] (v3) -- (v4);
    \draw [black, thick] (v4) -- (v5);
    \draw [black, thick] (v5) -- (v6);
    \draw [black, thick] (v7) -- (v3);
    \draw [black, thick] (v4) -- (v6);
    \draw [black, thick] (v4) -- (v7);

\end{tikzpicture}
\caption{After}
\end{subfigure}
\caption{An application of $\textbf{gcds}(\cdot,\{v_1,\;v_2\})$ to a finite graph.}\label{fig:gcdsmovegraph}
\end{figure}

As seen in the example in Figure \ref{fig:gcdsmovegraph}, an application of \textbf{gcds} to a graph with edges always creates a new pair of isolated points. Figure \ref{fig:gcdsmovegraphex2} illustrates that applying \textbf{gcds} may create more than two new isolated points. 

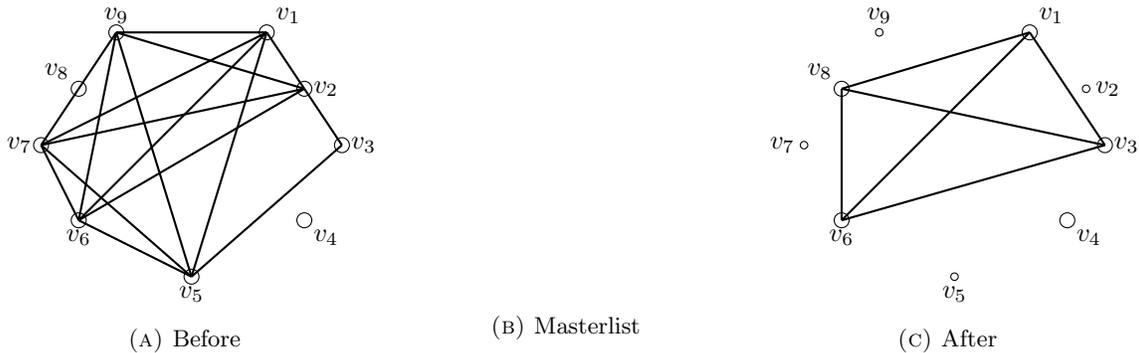
\begin{figure}[ht]
\centering
\begin{subfigure}{0.3\textwidth}
\begin{tikzpicture}[scale = .5]

\coordinate (v1) at (2,3);
\coordinate (v2) at (3,1.5);
\coordinate (v3) at (4,0);
\coordinate (v4) at (3,-2.0);
\coordinate (v5) at (0,-3.5);
\coordinate (v6) at (-3,-2.0);
\coordinate (v7) at (-4,0);
\coordinate (v8) at (-3,1.5);   
\coordinate (v9) at (-2,3);    

    \node[draw,circle, inner sep=2pt] at (v1){};
    \node[draw,circle, inner sep=2pt] at (v2){};
    \node[draw,circle, inner sep=2pt] at (v3){};
    \node[draw,circle, inner sep=2pt] at (v4){};
    \node[draw,circle, inner sep=2pt] at (v5){};
    \node[draw,circle, inner sep=2pt] at (v6){};
    \node[draw,circle, inner sep=2pt] at (v7){};
    \node[draw,circle, inner sep=2pt] at (v8){};
    \node[draw,circle, inner sep=2pt] at (v9){};

    \node at (v1) [above right]{$v_1$}; 
    \node at (v2) [right]{$v_2$}; 
    \node at (v3) [right]{$v_3$}; 
    \node at (v4) [below right]{$v_4$}; 
    \node at (v5) [below]{$v_5$}; 
    \node at (v6) [below]{$v_6$}; 
    \node at (v7) [left ]{$v_7$}; 
    \node at (v8) [above left]{$v_8$}; 
    \node at (v9) [above ]{$v_9$}; 
    
    \draw [black, thick] (v1) -- (v2);
    \draw [black, thick] (v1) -- (v5);
    \draw [black, thick] (v1) -- (v6);
    \draw [black, thick] (v1) -- (v7);
    \draw [black, thick] (v1) -- (v9);
    \draw [black, thick] (v2) -- (v3);
    \draw [black, thick] (v2) -- (v6);
    \draw [black, thick] (v2) -- (v7);
    \draw [black, thick] (v2) -- (v9);
    \draw [black, thick] (v3) -- (v5);
    \draw [black, thick] (v5) -- (v6);
    \draw [black, thick] (v5) -- (v7);
    \draw [black, thick] (v5) -- (v9);
    \draw [black, thick] (v6) -- (v7);
    \draw [black, thick] (v6) -- (v9);
    \draw [black, thick] (v7) -- (v8);
    \draw [black, thick] (v8) -- (v9);

\end{tikzpicture}
\caption{Before} 
\end{subfigure}
\begin{subfigure}{0.3\textwidth}
\centering
\vspace{0.7in}
\begin{tikzpicture}[scale = .5]
\begin{tabular}{|l|l|} \hline
${\mathbf v_2}$ & ${\mathbf v_7}$ \\ \hline
$v_1$ & $v_1$ \\ \hline
$v_3$ & $v_5$ \\ \hline
$v_6$ & $v_6$ \\ \hline
$v_9$ & $v_8$ \\ \hline
\end{tabular}
\end{tikzpicture}
\vspace{0.8in}
\caption{Masterlist}\label{fig:Masterlist}
\end{subfigure}
\begin{subfigure}{0.3\textwidth}
\centering
\begin{tikzpicture}[scale = .5]
\coordinate (v1) at (2,3);
\coordinate (v2) at (3.5,1.5);
\coordinate (v3) at (4,0);
\coordinate (v4) at (3,-2.0);
\coordinate (v5) at (0,-3.5);
\coordinate (v6) at (-3,-2.0);
\coordinate (v7) at (-4,0);
\coordinate (v8) at (-3,1.5);   
\coordinate (v9) at (-2,3);    

    \node[draw,circle, inner sep=2pt] at (v1){};
    \node[draw, circle, inner sep=1pt] at (v2){}; 
    \node[draw,circle, inner sep=2pt] at (v3){};
    \node[draw,circle, inner sep=2pt] at (v4){};
    \node[draw,circle, inner sep=1pt] at (v5){}; 
    \node[draw,circle, inner sep=2pt] at (v6){};
    \node[draw,circle, inner sep=1pt] at (v7){}; 
    \node[draw,circle, inner sep=2pt] at (v8){};
    \node[draw,circle, inner sep=1pt] at (v9){}; 

    \node at (v1) [above right]{$v_1$}; 
    \node at (v2) [right]{$v_2$}; 
    \node at (v3) [right]{$v_3$}; 
    \node at (v4) [below right]{$v_4$}; 
    \node at (v5) [below]{$v_5$}; 
    \node at (v6) [below]{$v_6$}; 
    \node at (v7) [left ]{$v_7$}; 
    \node at (v8) [above left]{$v_8$}; 
    \node at (v9) [above ]{$v_9$}; 
    
    \draw [black, thick] (v1) -- (v3);
    \draw [black, thick] (v1) -- (v6);
    \draw [black, thick] (v1) -- (v8);
    \draw [black, thick] (v3) -- (v6);
    \draw [black, thick] (v3) -- (v8);
    \draw [black, thick] (v6) -- (v8);

\end{tikzpicture}
\caption{After}
\end{subfigure}
\caption{Application of $\textbf{gcds}(\cdot,\{v_2,\;v_7\})$ to a finite graph, creating four new isolated vertices.}\label{fig:gcdsmovegraphex2}
\end{figure}

\vspace{0.1in}

Let $\mathcal{M}_n$ be the set of graphs with vertices the set of $n-1$ pointers associated with elements of $\textsf{S}_n$. Let 
  $ \Gamma_n:\textsf{S}_n \longrightarrow \mathcal{M}_n$
be the map that associates with each permutation $\alpha\in\textsf{S}_n$ its corresponding move graph. We now arrive at our main motivation for introducing \textbf{gcds}: 

\begin{theorem}\label{movegraphcds}
Let $n$ be a positive integer. For any permutation $\alpha\in\textsf{S}_n$ and for any interlocking pair of pointers $p$ and $q$ in $\alpha$, $\Gamma_n(\textbf{cds}_{p,q}(\alpha)) = \textbf{gcds}(\Gamma_n(\alpha),\{p,\; q\})$.
\end{theorem}
\begin{proof} 
Let $\alpha$ be a permutation.  Let $x,\; (x+1),\; y,\; (y+1)$ be entries in $\alpha$.  Let $p=(x,x+1)$ be the common pointer of the entries $x$ and $x+1$; let $q=(y,y+1)$ be the common pointer of the entries $y$ and $y+1$. Assume that $p$ and $q$ interlock in $\alpha$. Thus, $\{p,q\}$ is an edge in $\Gamma_{n}(\alpha)$. 

Writing the pointers in $\alpha$ in their order of occurrence gives the structure
$\underbrace{\hspace{.5cm}}_{A} p, \underbrace{\hspace{.5cm}}_{B}, q,  
   \underbrace{\hspace{.5cm}}_{C}, p, \underbrace{\hspace{.5cm}}_{D}, q, 
   \underbrace{\hspace{.5cm}}_{E}$,
 where the capital letters mark blocks of pointers. Doing the same with the pointers of  $\alpha' = \textbf{cds}_{p,q}(\alpha)$ gives the structure 
$\underbrace{\hspace{.5cm}}_{A} p,p,
   \underbrace{\hspace{.5cm}}_{D}, \underbrace{\hspace{.5cm}}_{C},
   \underbrace{\hspace{.5cm}}_{B}, q, q, \underbrace{\hspace{.5cm}}_{E}$.
\newline Let $r$ and $s$ be two pointers in $\alpha$, where the first instance of r is to the left of the first instance of s.  We will show that, regardless of the position of $r$ and $s$ relative to $p$ and $q$ in $\alpha$, $\{r,s\}$ is an edge in $\Gamma_n(\textbf{cds}_{p,q}(\alpha))$ if, and only if, $\{r,s\}$ is an edge in $\textbf{gcds}(\Gamma_n(\alpha),\{p,\; q\})$.

We treat the case when $r$ and $s$ interlock in $\alpha$, leaving the other case to the reader. Thus, assume that $r$ and $s$ interlock in $\alpha$. Thus, $\{r,\; s\}$ is an edge in $\Gamma_n(\alpha)$.  Let $r_{1},s_{1}$ denote the left instances of $r$ and $s$;  let $r_{2},s_{2}$ denote the right instances of $r$ and $s$.  So these pointers appear in the form 
$r_{1},\ldots s_{1},\ldots r_{2},\ldots s_{2}$.
There are seven relevant cases for how $r$ and $s$ relate to $p$ and $q$: 
\vspace{0.1in}

In Cases I-III, $s$ interlocks neither $p$ nor $q$.  So neither$\{p,s\}$ nor $\{q,s\}$ is an edge in $\Gamma_{n}(\alpha)$.  Therefore, $\{r,s\}$ is an edge in $\textbf{gcds}(\Gamma_n(\alpha),\{p,\; q\})$.  

{\flushleft Case I:} r and s interlock neither $p$ nor $q$.  Therefore, $r$ and $s$ could be contained in two possible blocks of $\alpha$:
\newline (a) All instances of $r$ and $s$ are found in $A$ or $E$.
\newline (b) All instances of $r$ and $s$ are found in one block.  
\begin{quote} In both I(a) and I(b) rearrangement of the blocks $B$, $C$, and $D$ does not affect the arrangement of $r$ and $s$, and so in $\alpha'$ $r$ and $s$ still interlock.  
Thus, in Case I, $\{r,s\}$ is an edge in $\Gamma_{n}(\alpha')$.
\end{quote}

{\flushleft Case II:} $s$ interlocks with neither $p$ nor $q$, while $r$ interlocks with $p$ but not with $q$.
\newline (a) $r_{1}$ is in $A$ and $r_{2}$ is in $B$.
\begin{quote} $s_1$ is in $A$ or $B$: If $s_1$ is in $A$, then $s_2$ is in $E$, while if $s_1$ is in $B$, then $s_2$ is in $B$.  
\newline Blocks $A,B$, and $E$ do not change order relative to each other, so in either situation $r$ and $s$ still interlock in $\alpha'$.    
\end{quote}
(b) $r_{1}$ is in $C$ and $r_{2}$ is in $D$.
\begin{quote} So $s_1$ is in $D$: Therefore $s_2$ is in $D$.  As blocks $C$ and $D$ switch order, in $\alpha'$ $r$ and $s$ are arranged as follows: $s_1$, $r_2,s_2,r_1$. 
\newline Then in Case II, $r$ and $s$ remain interlocked in $\alpha'$, whence $\{r,s\}$ is an edge in $\Gamma_{n}(\alpha')$. 
\end{quote}

{\flushleft Case III:} s interlocks with neither $p$ nor $q$, while $r$ interlocks with both $p$ and $q$.
 \newline (a) $r_{1}$ is in $A$ and $ r_{2}$ is in $C$. \begin{quote} Now $s_1$ is in $A$ or $C$. If $s_1$ is in $A$, then $s_2$ is in $E$, while if $s_1$ is in $C$, then $s_2$ is in $C$. \end{quote}  (b) $r_{1}$ is in C and $r_2$ is in E. \begin{quote} So $s_1$ and $s_2$ are in E. \newline Since the order of $A, C$, and $E$ is unaffected by $\textbf{cds}_{p,q}(\alpha)$, $r$ and $s$ interlock in $\alpha'$.\end{quote}Therefore, in Case III, $r$ and $s$ interlock in $\alpha'$ whence $\{r,s\}$ is an edge in $\Gamma_{n}(\alpha')$.
\vspace{0.1in}

In Cases IV-VI, both $r$ and $s$ interlock $p$ or $q$.  

{\flushleft Case IV:} both $r$ and $s$ interlock with $p$ but both do not interlock with $q$.
In $\Gamma_n(\alpha)$ $s$ and $r$ both have an edge with $p$, but neither have an edge with $q$.  So in $\textbf{gcds}(\Gamma_n(\alpha),\{p,\; q\})$, $\{r,s\}$ will be an edge.
\newline Now in the permutation $\alpha$ we have one of the following two cases: 
\newline (a) $r_{1}$ is in $A$ and $r_2$ is in B.  So $s_{1}$ is in $A$ or B. 
\begin{quote} If $s_1$ is in A, then $s_2$ is in B.\end{quote} 
\begin{quote} If $s_1$ is in B, then $s_2$ is in E.\end{quote}
\begin{quote} The order of the pointers is unchanged in $\alpha'$ as the order of $A$, $B$ and $E$ is unchanged.
\end{quote}
(b)  $r_{1}$ is in $C$ and $r_{2}$ is in $D$.  So $s_1$ is in C and $s_{2}$ is in $D$.
\begin{quote} The pointers have form $r_2$,\ldots, $s_2,\ldots,r_1,\ldots,s_1$ in $\alpha'$. 
\newline Thus, in Case IV, $r$ and $s$ interlock in $\alpha'$, meaning $\{r,s\}$ is an edge in $\Gamma_{n}(\alpha')$.
\end{quote}

{\flushleft Case V:} $r$ interlocks with $p$ but not $q$, and $s$ interlocks with $q$ but not $p$.
\newline In $\Gamma_{n}(\alpha)$, $\{r,p\}$ and $\{s,q\}$ are edges, so r and s both appear in the masterlist.  However, r and s do not appear in the same column of the master list, and so in $\textbf{gcds}(\Gamma_n(\alpha),\{p,\; q\})$, $\{r,s\}$ is not an edge.
\newline (a) $r_{1}$ is in $A$ and $r_2$ is in $B$.  So $s_1$ is in $A$ or $B$.
\begin{quote} If $s_1$ is in A, then $s_2$ is in D.  B and D reverse position, yielding $r_{1}$, $s_1$, $s_2$, $r_2$.
\end{quote}
\begin{quote} If $s_1$ is in B, then $s_2$ is in C.  B and C reverse position, yielding $r_{1}$, $s_2$, $s_1$, $r_2$.
\end{quote}
(b) $r_{1}$ is in $C$ and $r_2$ is in $D$.  So $s_1$ is in $D$ and $s_2$ is in $E$.
\begin{quote} $D$ and $C$ reverse position, yielding $s_{1}$, $r_2$, $r_1$, $s_2$.
\end{quote}
Then in Case V, $r$ and $s$ do not interlock in $\alpha'$, and so $\{r,s\}$ is not an edge in $\Gamma_{n}(\alpha')$.

{\flushleft Case VI:} $r$ interlocks with both $p$ and $q$, while $s$ interlocks with $p$ but not $q$.
\newline As $\{r,p\}$, $\{r,q\}$ and $\{s,p\}$ all are edges in $\Gamma_n(\alpha)$, $r$ and $s$ appear an odd number of times and in both columns of the Masterlist of $p$ and $q$.
Therefore, $\{r,s\}$ is not an edge in $\textbf{gcds}(\Gamma_n(\alpha),\{p,\; q\})$.
\newline In the permutation $\alpha: \newline r_{1}$ must be in $A$ and $r_2$ must be in $C$ \begin{quote}Then $s_1$ is in B or C.  \newline (a) If $s_1$ is in $B$, then $s_2$ is in $E$.  So in $\alpha'$, the order will be $r_{1},r_{2},s_{1},s_{2}$. \newline (b) If $s_1$ is in C, then $s_2$ is in D. So in $\alpha'$, the order will be $r_{1},s_{2},s_{1},r_{2}$.\end{quote} 
Then in Case VI, $r$ and $s$ do not interlock in $\alpha'$, and so $\{r,s\}$ is not an edge in $\Gamma_{n}(\alpha')$.

{\flushleft Case VII:}  Each of $r$ and $s$ interlocks both $p$ and $q$.
\newline Then $\{r,p\}$, $\{r,q\}$, $\{s,p\}$ and $\{s,q\}$ all are edges in $\Gamma_n(\alpha)$. Thus $r$ and $s$ appear an even number of times and in both columns of the Masterlist of $p$ and $q$ in $\Gamma_n(\alpha)$.  Therefore, $\{r,s\}$ is still an edge in $\textbf{gcds}(\Gamma_n(\alpha),\{p,\; q\})$.
\newline In the permutation $\alpha$ there are two possibilities for the placement of r:
\newline (a) $r_1$ is in $A$ and $r_2$ is in C.  So $s_1$ is in A or C.  \begin{quote} If $s_1$ is in A, then $s_2$ is in $C$.\end{quote} \begin{quote} If $s_1$ is in C, then $s_2$ is in E.\end{quote} (b) $r_1$ is in C and $r_2$ is in $E$, so $s_1$ is in C and $s_2$ are in $E$. \newline $C$ does not change its order relative to $E$ or $A$, so $r$ and $s$ interlock in $\alpha'$.  
\newline Thus, in Case VII, $\{r,s\}$ is an edge in $\Gamma_n(\alpha')$.
\vspace{0.1in}

We find that in all cases $\{r,\; s\}$ is an edge in $\Gamma_n(\alpha^{\prime})$ if, and only if, it is an edge in $\textbf{gcds}(\Gamma_n(\alpha),\{p,\; q\})$. This completes the proof of Theorem 2.
\end{proof}


\section{The \textbf{gcds} game on graphs}

For expository convenience we now define the following modification $\textbf{gcds}_2$ of \textbf{gcds}: Let $\mathcal{G}= (V,\;E)$ be a fixed finite graph and let $\{a,\; b\}$ be an edge of  $\mathcal{G}$. Compute the graph $(V,E^{\prime}) = \textbf{gcds}(\mathcal{G},\{a,\; b\})$. Delete a vertex $v\in V$ if $v$ appeared in the masterlist as \emph{not} the only vertex in the master list, and is isolated after the application of \textbf{gcds}. Let $V^{\prime}$ be the remaining set of vertices. Then we define $\textbf{gcds}_2(\mathcal{G},\{a,\; b\}):=(V^{\prime},\; E^{\prime})$.

To define the game $\textsf{GCDS}(\mathcal{G},\textsf{A})$, let \textsf{A} be a subset of the set of vertices of the graph $\mathcal{G}$. Suppose after a number of rounds the current graph is 
$\mathcal{C} = (V^{\prime},\; E^{\prime})$. Then the next player to play selects an edge $\{a,\; b\}$ of $\mathcal{C}$, and plays the move $\textbf{gcds}_2(\mathcal{C},\{a,\; b\})$. The game starts with ONE making such a move for the initial graph $\mathcal{G}$. The players alternate in making moves until there are no edges left.

Let $\mathcal{G}^{\prime}$ denote the graph when the game ends\footnote{$\mathcal{G}^{\prime}$ has an empty set of edges.}. If the set of vertices of $\mathcal{G}^{\prime}$ is nonempty and a subset of \textsf{A}, then ONE wins. Else, TWO wins.

As an example, consider the game on the following graph $\mathcal{G} = (V,\; E)$:
\begin{itemize}
\item{$V = \{1,\; 2,\; 3,\; 4,\; 5,\; 6,\; 7\}$}
\item{$E = \{\{1,\; 2\} ,\; \{1,\; 3\} ,\; \{1,\; 4\} ,\; \{1,\; 7\} ,\; \{2,\; 3\} ,\; \{2,\; 4\} ,\;\{2,\; 6\} ,\; \{3,\; 4\} ,\;  \{3,\; 6\} ,\; \{4,\; 5\} ,\; \{5,\; 6\} ,\; \{6,\; 7\} \}$.}
\end{itemize}
The set of vertices favorable to ONE is $\textsf{A} = \{1,\; 3,\; 5,\; 7\}$. Figure \ref{fig:graphcdsgame} (A) depicts the graph $\mathcal{G}$ with vertices in \textsf{A} marked with solid circles. (B), (C) and (D) depict the results of moves by ONE and TWO. The game terminates with ONE winning the play as the sole surviving vertex belongs to \textsf{A}.

\begin{figure}[h]
\begin{subfigure}{0.22\textwidth}
\centering
\begin{tikzpicture}[scale = .4]
\coordinate (v1) at (3,0);
\coordinate (v2) at (2,-2);
\coordinate (v3) at (-1,-3);
\coordinate (v4) at (-3,-1);
\coordinate (v5) at (-3,1);
\coordinate (v6) at (-1,3);
\coordinate (v7) at (2,2);
    \node[draw,black, very thick, circle, inner sep=2pt] at (v1){};
    \node[draw,circle, densely dotted, very thick, inner sep=2pt] at (v2){};
    \node[draw,circle, very thick, inner sep=2pt] at (v3){};
    \node[draw,circle, densely dotted, very thick, inner sep=2pt] at (v4){};
    \node[draw,circle, very thick, inner sep=2pt] at (v5){};
    \node[draw,circle, densely dotted, very thick, inner sep=2pt] at (v6){};
    \node[draw,circle, very thick, inner sep=2pt] at (v7){};
    \node at (v1) [right]{1};
    \node at (v2) [below right]{2};
    \node at (v3) [below]{3};
    \node at (v4) [below left]{4};
    \node at (v5) [left]{5};
    \node at (v6) [above left]{6};
    \node at (v7) [above ]{7};    
    \draw [black, thick] (v1) -- (v2);
    \draw [black, thick] (v2) -- (v3);
    \draw [black, thick] (v3) -- (v4);
    \draw [black, thick] (v4) -- (v5);
    \draw [black, thick] (v5) -- (v6);
    \draw [black, thick] (v6) -- (v7);
    \draw [black, thick] (v3) -- (v1);
    \draw [black, thick] (v4) -- (v2);
    \draw [black, thick] (v6) -- (v3);
    \draw [black, thick] (v4) -- (v1);
    \draw [black, thick] (v6) -- (v2);
    \draw [black, thick] (v6) -- (v3);
    \draw [black, thick] (v7) -- (v1);
\end{tikzpicture}
\caption{$\mathcal{G} = (V,\; E)$.}
\end{subfigure}
\begin{subfigure}{0.22\textwidth}
\centering
\begin{tikzpicture}[scale = .4]
\coordinate (v1) at (3,0);
\coordinate (v2) at (2,-2);
\coordinate (v3) at (-1,-3);
\coordinate (v4) at (-3,-1);
\coordinate (v5) at (-3,1);
\coordinate (v6) at (-1,3);
\coordinate (v7) at (2,2);
    \node[draw,circle, densely dotted, very thick,  inner sep=2pt] at (v1){};
    \node[draw,circle, very thick, inner sep=2pt] at (v3){};
    \node[draw,circle, very thick, inner sep=2pt] at (v5){};
    \node[draw,circle, densely dotted, very thick, inner sep=2pt] at (v6){};
    \node[draw,circle, very thick, inner sep=2pt] at (v7){};
    \node at (v1) [right]{1};
    \node at (v3) [below]{3};
    \node at (v5) [left]{5};
    \node at (v6) [above left]{6};
    \node at (v7) [above ]{7};
        \draw [black, thick] (v3) -- (v5);
    \draw [black, thick] (v5) -- (v1);
    \draw [black, thick] (v6) -- (v7);
    \draw [black, thick] (v3) -- (v1);
    \draw [black, thick] (v6) -- (v1);
    \draw [black, thick] (v7) -- (v1);
\end{tikzpicture}
\caption{ONE's move: $\mathcal{O}_1$.}
\end{subfigure}
\begin{subfigure}{0.22\textwidth}
\centering
\begin{tikzpicture}[scale = .4]
\coordinate (v1) at (3,0);
\coordinate (v2) at (2,-2);
\coordinate (v3) at (-1,-3);
\coordinate (v4) at (-3,-1);
\coordinate (v5) at (-3,1);
\coordinate (v6) at (-1,3);
\coordinate (v7) at (2,2);

    \node[draw,circle, very thick, inner sep=2pt] at (v5){};
    \node[draw,circle, densely dotted, very thick, inner sep=2pt] at (v6){};
    \node[draw,circle, very thick, inner sep=2pt] at (v7){};
      \node at (v3)[below]{}; 
    \node at (v5) [left]{5};
    \node at (v6) [above left]{6};
    \node at (v7) [above ]{7};
    \draw [black, thick] (v7) -- (v5);
    \draw [black, thick] (v5) -- (v6);
    \draw [black, thick] (v6) -- (v7);
\end{tikzpicture}
\caption{\small TWO's move: $\mathcal{T}_1$.}
\end{subfigure}
\begin{subfigure}{0.22\textwidth}
\centering
\begin{tikzpicture}[scale = .4]
\coordinate (v1) at (3,0);
\coordinate (v2) at (2,-2);
\coordinate (v3) at (-1,-3);
\coordinate (v4) at (-3,-1);
\coordinate (v5) at (-3,1);
\coordinate (v6) at (-1,3);
\coordinate (v7) at (2,2);

    \node[draw,circle, very thick, inner sep=2pt] at (v5){};
    \node at (v1) [right]{}; 
    \node at (v2) [below right]{}; 
    \node at (v3)[below]{}; 
    \node at (v4) [below left]{}; 
    \node at (v5) [left]{5};
    \node at (v6) [above left]{}; 
    \node at (v7) [above ]{}; 
\end{tikzpicture}
\caption{ONE's move: $\mathcal{O}_2$.}
\end{subfigure}
\caption{(A) depicts the original graph. Solid circles denote vertices favorable to ONE. (B) ONE moves $\{2,\; 4\}$. (C) TWO responds with $\{1,\; 3\}$. (D) ONE responds with $\{6,\; 7\}$.}\label{fig:graphcdsgame}
\end{figure}
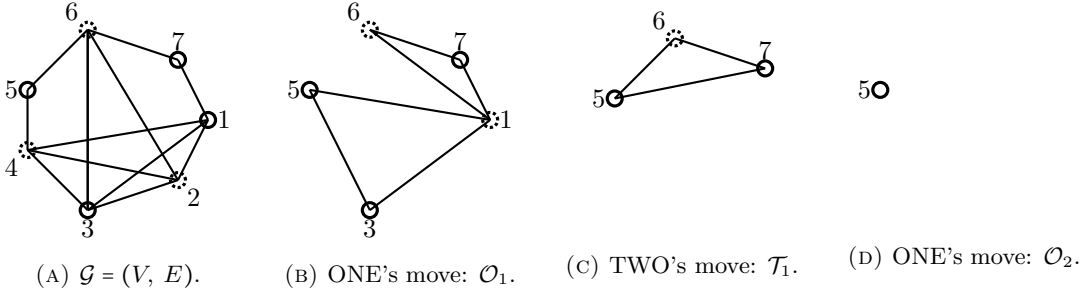

\begin{theorem}\label{cdsgraph}
For each finite graph $\mathcal{G} = (V,\; E)$ and each set $\textsf{A}\subseteq V$, 
the game $\emph{\textsf{GCDS}}(\mathcal{G},\textsf{A})$ is determined\footnote{A game is determined if one of the players has a winning strategy.}.
\end{theorem}
\begin{proof} This fact follows from Zermelo's classical theorem that in a finite win-loose two-player game of perfect information one of the players has a winning strategy \cite{Z}. 
\end{proof}

Theorem \ref{cdsgraph} raises the following decision problem:\\
\vspace{0.2cm}

\begin{center}
\parbox{0.8\textwidth}
{
{\flushleft{\bf GRAPH GAME $\textsf{GCDS}$:}}\\
{\flushleft{INSTANCE:}} A finite graph $\mathcal{G}=(V,E)$ and a set \textsf{A}.  
{\flushleft{QUESTION:}} Does ONE have a winning strategy in the game $\textsf{GCDS}(\mathcal{G},\textsf{A})$? 
  }
\end{center}
\vspace{0.2cm}

{\flushleft{}}  We solved only a special case of this decision problem. Towards describing this result we introduce the following terminology and notation: Graphs $\mathcal{F}_1 = (V_1, E_1)$ and $\mathcal{F}_2 = (V_2, E_2)$ are \emph{isomorphic} if there is a bijective function $f:V_1\longrightarrow V_2$ such that for all $x$ and $y$ in $V_1$, $\{x,\; y\}\in E_1$ if, and only if, $\{f(x),\; f(y)\}\in E_2$. 

Recursively define the graphs $(\mathcal{G}_n:n\in \mathbb{N})$  as follows:
{\flushleft{$\mathcal{G}_1 = (V_1,\; E_1)$} where: $V_1 = \{1,\; 2,\; 3\}$
 and 
 $E_1 = \{\{1,\; 2\},\; \{1,\; 3\},\; \{2,\; 3\}\}$.

{\flushleft{Assume that $\mathcal{G}_n= (V_n,\; E_n)$ has been defined with $V_n = \{1,\; 2,\; \cdots,\; 2n+1\}$, and define $\mathcal{G}_{n+1} = (V_{n+1},\; E_{n+1})$ by }}
\begin{quote}
  $V_{n+1} = V_n\bigcup\{2n+2,\; 2n+3 \}$ 
and 
$E_{n+1} = E_n \bigcup \{\{2n+1,\; 2n+2\},\; \{2n+1,\; 2n+3\},\; \{2n+2,\; 2n+3\}\}.$
\end{quote}

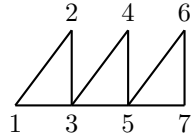
\begin{figure}[h]
\centering
\begin{tikzpicture}[scale = .5]

    \coordinate (v1) at (0,0);
    \coordinate (v2) at (1.5,0);
    \coordinate (v3) at (1.5,2);
    \coordinate (v4) at (3,0);
    \coordinate (v5) at (3,2);
    \coordinate (v6) at (4.5,0);
    \coordinate (v7) at (4.5,2);
    
    \node at (v1) [below]{1};
    \node at (v2) [below]{3};
    \node at (v3) [above]{2};
    \node at (v4) [below]{5};
    \node at (v5)  [above]{4};
    \node at (v6) [below]{7};
    \node at (v7)  [above]{6};
    
    \draw [black, thick] (v1) -- (v2);
    \draw  [black, thick]  (v1) -- (v3);
    \draw  [black, thick]  (v3) -- (v2);
    \draw [black, thick] (v2) -- (v4);
    \draw  [black, thick]  (v2) -- (v5);
    \draw  [black, thick]  (v4) -- (v5);
    \draw  [black, thick]  (v4) -- (v6);
    \draw  [black, thick]  (v4) -- (v7);
    \draw  [black, thick]  (v6) -- (v7);

\end{tikzpicture}
\caption{The graph $\mathcal{G}_3$. $\mathcal{G}_1$ consists of the leftmost triangle, and $\mathcal{G}_2$ of the leftmost two triangles}\label{fig:G1G2G3}
\end{figure}

\begin{theorem}\label{gcdsonclass}
Let $n>1$ be an integer. Let $\mathcal{F}$ be a graph isomorphic to $\mathcal{G}_n$. Let $\{u,\; v\}$ be an edge of $\mathcal{F}$. Then $\textbf{gcds}_2(\mathcal{F},\{u,\; v\})$ is isomorphic to $\mathcal{G}_{n-1}$.
\end{theorem}
\begin{proof} The proof consists of examining the result of $\textbf{gcds}_2(\mathcal{G}_n,\{j,\; k\})$ where $\{j,\; k\}$ is any of the three edges of the dashed triangle displayed in the generic segment of the graph $\mathcal{G}_n$ in Figure \ref{fig:genericblue}, and also cases where the edge $\{j,\; k\}$ is from the leading triangle or the trailing triangle of $\mathcal{G}_n$.
\begin{figure}[ht]
\centering
\fbox{\begin{tikzpicture}[scale = .6]

    \coordinate (v16) at (-2,1);
    \coordinate (v17) at (-1.5,1);
    \coordinate (v18) at (-1,1);

    \coordinate (v1) at (0,0);
    \coordinate (v2) at (1.5,0);
    \coordinate (v3) at (1.5,2);
    \coordinate (v4) at (3,0);
    \coordinate (v5) at (3,2);
    \coordinate (v6) at (4.5,0);
    \coordinate (v7) at (4.5,2);
    \coordinate (v8) at (9,0);
    \coordinate (v9) at (9,2);
    \coordinate (v10) at (7.5,0);
    \coordinate (v11) at (6,0);
    \coordinate (v12) at (6,2);
    \coordinate (v13) at (6.5,1);
    \coordinate (v14) at (7,1);
    \coordinate (v15) at (7.5,1);
    
    \node at (v1) [below]{\small 2i+1};
    \node at (v2) [below]{\small 2i+3};
    \node at (v3) [above]{\small 2i+2}; 
    \node at (v4) [below]{\small 2i+5};
    \node at (v5)  [above]{\small 2i+4};
    \node at (v6) [below]{\small 2i+7};
    \node at (v7)  [above]{\small 2i+6};
    \node at (v8) [below]{\small 2n+1};
    \node at (v9)  [above]{\small 2n};
    \node at (v10) [below]{\small 2n-1};
    \node at (v11) [below]{\small 2i+9};
    \node at (v12)  [above]{\small 2i+8};
    \node [draw,circle, inner sep=1pt, fill] at (v13){};
    \node[ draw,circle, inner sep=1pt, fill] at (v14){};
    \node[draw,circle, inner sep=1pt, fill] at (v15){};
    
  \node[draw,circle, inner sep=1pt, fill] at (v16){};
    \node[draw,circle, inner sep=1pt, fill] at (v17){};
    \node[draw,circle, inner sep=1pt, fill] at (v18){};
  
    \draw [black, thick] (v1) -- (v2);
    \draw  [black, thick]  (v1) -- (v3);
    \draw  [black, thick]  (v3) -- (v2);
     \draw [black, thick] (v2) -- (v4);
    \draw  [black, thick]  (v2) -- (v5);
    \draw  [black, thick]  (v5) -- (v4) ;
     \draw [black, thick, dashed] (v4) -- (v6);
    \draw  [black, thick, dashed]  (v4) -- (v7);
    \draw  [black, thick, dashed]  (v6) -- (v7) ;
    \draw  [black, thick]  (v8) -- (v9);
    \draw  [black, thick]  (v10) -- (v9);
    \draw  [black, thick]  (v8) -- (v10);
    \draw  [black, thick]  (v11) -- (v12);
    \draw  [black, thick]  (v11) -- (v6);
    \draw  [black, thick]  (v12) -- (v6);
   
\end{tikzpicture}}
\caption{}\label{fig:genericblue}
\end{figure}
\vspace{0.1in}

Figures \ref{fig:GenericSegment56}, \ref{fig:GenericSegment57} and \ref{fig:GenericSegment67} each illustrates the resulting graph when an edge of the triangle marked in dashed lines is targeted.  We leave examination of the remaining cases to the reader.
\end{proof}

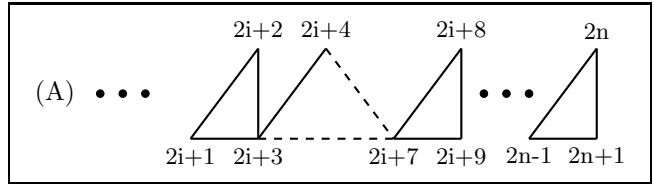
\begin{figure}[ht]
\fbox{
 \centering
 \begin{tikzpicture}[scale = .6]
    \coordinate (v19) at (-3,1);
    \coordinate (v16) at (-2,1);
    \coordinate (v17) at (-1.5,1);
    \coordinate (v18) at (-1,1);
    \coordinate (v1) at (0,0);
    \coordinate (v2) at (1.5,0);
    \coordinate (v3) at (1.5,2);
    \coordinate (v4) at (3,0);
    \coordinate (v5) at (3,2);
    \coordinate (v6) at (4.5,0);
    \coordinate (v7) at (4.5,2);
    \coordinate (v8) at (9,0);
    \coordinate (v9) at (9,2);
    \coordinate (v10) at (7.5,0);
    \coordinate (v11) at (6,0);
    \coordinate (v12) at (6,2);
    \coordinate (v13) at (6.5,1);
    \coordinate (v14) at (7,1);
    \coordinate (v15) at (7.5,1);
    \node at (v1) [below]{\small 2i+1};
    \node at (v2) [below]{\small 2i+3};
    \node at (v3) [above]{\small 2i+2}; 
    \node at (v5)  [above]{\small 2i+4};
    \node at (v6) [below]{\small 2i+7};
    \node at (v8) [below]{\small 2n+1};
    \node at (v9)  [above]{\small 2n};
    \node at (v10) [below]{\small 2n-1};
    \node at (v11) [below]{\small 2i+9};
    \node at (v12)  [above]{\small 2i+8};
    \node [draw,circle, inner sep=1pt, fill] at (v13){};
    \node[ draw,circle, inner sep=1pt, fill] at (v14){};
    \node[draw,circle, inner sep=1pt, fill] at (v15){};    
   \node[draw,circle, inner sep=1pt, fill] at (v16){};
    \node[draw,circle, inner sep=1pt, fill] at (v17){};
    \node[draw,circle, inner sep=1pt, fill] at (v18){};  
   \node at (v19){(A)};
    \draw [black, thick] (v1) -- (v2);
    \draw  [black, thick]  (v1) -- (v3);
    \draw  [black, thick]  (v3) -- (v2);
    \draw  [black, thick]  (v2) -- (v5);
     \draw [black, thick, dashed] (v2) -- (v6);
    \draw  [black, thick, dashed]  (v5) -- (v6);
    \draw  [black, thick]  (v8) -- (v9);
    \draw  [black, thick]  (v10) -- (v9);
    \draw  [black, thick]  (v8) -- (v10);
    \draw  [black, thick]  (v11) -- (v12);
    \draw  [black, thick]  (v11) -- (v6);
    \draw  [black, thick]  (v12) -- (v6);
\end{tikzpicture} }
\caption{$\textbf{gcds}_2(\mathcal{G}_n,\{2i+5,\; 2i+6\})$.  Dashed lines denote new edges.} \label{fig:GenericSegment56} 
\end{figure}
\vspace{0.1in}

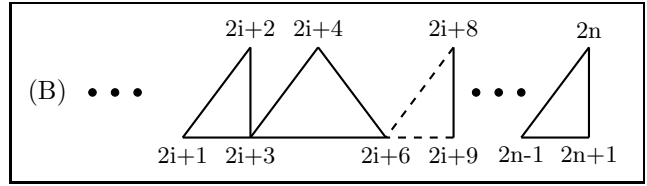
\begin{figure}[ht]
\fbox{\centering
\begin{tikzpicture}[scale = .6]
    \coordinate (v19) at (-3,1);
    \coordinate (v16) at (-2,1);
    \coordinate (v17) at (-1.5,1);
    \coordinate (v18) at (-1,1);
    \coordinate (v1) at (0,0);
    \coordinate (v2) at (1.5,0);
    \coordinate (v3) at (1.5,2);
    \coordinate (v4) at (3,0);
    \coordinate (v5) at (3,2);
    \coordinate (v6) at (4.5,0);
    \coordinate (v7) at (4.5,0); 
    \coordinate (v8) at (9,0);
    \coordinate (v9) at (9,2);
    \coordinate (v10) at (7.5,0);
    \coordinate (v11) at (6,0);
    \coordinate (v12) at (6,2);
    \coordinate (v13) at (6.5,1);
    \coordinate (v14) at (7,1);
    \coordinate (v15) at (7.5,1);
    \node at (v1) [below]{\small 2i+1};
    \node at (v2) [below]{\small 2i+3};
    \node at (v3) [above]{\small 2i+2}; 
    \node at (v5)  [above]{\small 2i+4};
    \node at (v7)  [below]{\small 2i+6}; 
    \node at (v8) [below]{\small 2n+1};
    \node at (v9)  [above]{\small 2n};
    \node at (v10) [below]{\small 2n-1};
    \node at (v11) [below]{\small 2i+9};
    \node at (v12)  [above]{\small 2i+8};
    \node [draw,circle, inner sep=1pt, fill] at (v13){};
    \node[ draw,circle, inner sep=1pt, fill] at (v14){};
    \node[draw,circle, inner sep=1pt, fill] at (v15){};    
    \node[draw,circle, inner sep=1pt, fill] at (v16){};
    \node[draw,circle, inner sep=1pt, fill] at (v17){};
    \node[draw,circle, inner sep=1pt, fill] at (v18){};
\node at (v19){(B)};
    \draw [black, thick] (v1) -- (v2);
    \draw  [black, thick]  (v1) -- (v3);
    \draw  [black, thick]  (v3) -- (v2);
     \draw [black, thick] (v2) -- (v7);
    \draw  [black, thick]  (v2) -- (v5);
    \draw  [black, thick]  (v5) -- (v7) ;
    \draw  [black, thick, dashed]  (v12) -- (v7);
    \draw  [black, thick, dashed]  (v11) -- (v7) ;
    \draw  [black, thick]  (v8) -- (v9);
    \draw  [black, thick]  (v10) -- (v9);
    \draw  [black, thick]  (v8) -- (v10);
    \draw  [black, thick]  (v11) -- (v12);
\end{tikzpicture} }
\caption{$\textbf{gcds}_2(\mathcal{G}_n, \{2i+5,\; 2i+7\})$.  Dashed lines denote new edges. Vertex $2i+6$ is moved down for clarity.}\label{fig:GenericSegment57} 

\vspace{0.1in}
\end{figure}

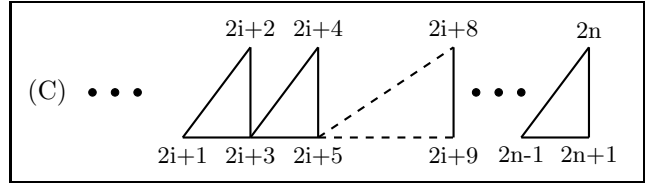
\begin{figure}[ht]
\fbox{\centering
\begin{tikzpicture}[scale = .6]
    \coordinate (v19) at (-3,1);
    \coordinate (v16) at (-2,1);
    \coordinate (v17) at (-1.5,1);
    \coordinate (v18) at (-1,1);
    \coordinate (v1) at (0,0);
    \coordinate (v2) at (1.5,0);
    \coordinate (v3) at (1.5,2);
    \coordinate (v4) at (3,0);
    \coordinate (v5) at (3,2);
    \coordinate (v6) at (4.5,0);
    \coordinate (v7) at (4.5,2);
    \coordinate (v8) at (9,0);
    \coordinate (v9) at (9,2);
    \coordinate (v10) at (7.5,0);
    \coordinate (v11) at (6,0);
    \coordinate (v12) at (6,2);
    \coordinate (v13) at (6.5,1);
    \coordinate (v14) at (7,1);
    \coordinate (v15) at (7.5,1);
    \node at (v1) [below]{\small 2i+1};
    \node at (v2) [below]{\small 2i+3};
    \node at (v3) [above]{\small 2i+2}; 
    \node at (v4) [below]{\small 2i+5};
    \node at (v5)  [above]{\small 2i+4};
    \node at (v8) [below]{\small 2n+1};
    \node at (v9)  [above]{\small 2n};
    \node at (v10) [below]{\small 2n-1};
    \node at (v11) [below]{\small 2i+9};
    \node at (v12)  [above]{\small 2i+8};
    \node [draw,circle, inner sep=1pt, fill] at (v13){};
    \node[ draw,circle, inner sep=1pt, fill] at (v14){};
    \node[draw,circle, inner sep=1pt, fill] at (v15){};   
   \node[draw,circle, inner sep=1pt, fill] at (v16){};
    \node[draw,circle, inner sep=1pt, fill] at (v17){};
    \node[draw,circle, inner sep=1pt, fill] at (v18){};  
\node at (v19){(C)};
    \draw [black, thick] (v1) -- (v2);
    \draw  [black, thick]  (v1) -- (v3);
    \draw  [black, thick]  (v3) -- (v2);
     \draw [black, thick] (v2) -- (v4);
    \draw  [black, thick]  (v2) -- (v5);
    \draw  [black, thick]  (v5) -- (v4) ;
     \draw [black, thick, dashed] (v4) -- (v12);
    \draw  [black, thick, dashed]  (v4) -- (v11);
    \draw  [black, thick]  (v8) -- (v9);
    \draw  [black, thick]  (v10) -- (v9);
    \draw  [black, thick]  (v8) -- (v10);
    \draw  [black, thick]  (v11) -- (v12);
\end{tikzpicture} }
\caption{$\textbf{gcds}_2(\mathcal{G}_n\{2i+6,\; 2i+7\})$.  Dashed lines denote new edges.}  \label{fig:GenericSegment67}
\end{figure}

For a finite graph $\mathcal{G} = (V,\;E)$ and a set $\textsf{A}$ the pair $(\mathcal{G}, \textsf{A})$ is said to be a \emph{position}. Two positions $(\mathcal{F}_1,\textsf{B}_1)$ and $(\mathcal{F}_2,\textsf{B}_2)$ are \emph{isomorphic} if there is an isomorphism $f$ from graph $\mathcal{F}_1$ to graph $\mathcal{F}_2$ such that $\textsf{B}_2 = \{f(b):\; b\in \textsf{B}_1\}$.
A position $(\mathcal{G},\textsf{A})$ is said to be a \textsf{P}-\emph{position} if there is a winning strategy from this position for the player whose move produced this position. It is said to be a \textsf{N}-\emph{position} if the next player to move has a winning strategy from this position.  

For each $n$ define the set $\textsf{A}_n = \{2\}\bigcup \{ k \le 2n+1: k \mod 4 = 0 \}$, a subset of the set of vertices of $\mathcal{G}_n$. 
Figure \ref{fig:specialgraphgeneric} depicts a generic example of a position $(\mathcal{G}_n,\textsf{A}_n)$, with members of $\textsf{A}_n$ marked in gray. From now on, we consider the class $\mathcal{T}$ of positions $(\mathcal{F},\textsf{B})$ for which there exists an $n$ such that $(\mathcal{F},\textsf{B})$ is isomorphic to $(\mathcal{G}_n,\textsf{A}_n)$.

 \begin{figure}[ht]
\centering
\begin{tikzpicture}[scale = .6]

    \coordinate (v1) at (0,0);
    \coordinate (v2) at (1.5,0);
    \coordinate (v3) at (1.5,2);
    \coordinate (v4) at (3,0);
    \coordinate (v5) at (3,2);
    \coordinate (v6) at (4.5,0);
    \coordinate (v7) at (4.5,2);
    \coordinate (v11) at (6,0);
    \coordinate (v12) at (6,2);

    \coordinate (v13) at (8.0,1);
    \coordinate (v14) at (8.5,1);
    \coordinate (v15) at (9.0,1);

    \coordinate (v10) at (9.0,0);
    \coordinate (v8) at (10.5,0);
    \coordinate (v9) at (10.5,2);
    
    \coordinate (v16) at (7.5,0);
    \coordinate (v17) at (7.5,2);

    \node at (v1) [below]{1};
    \node at (v2) [below]{3};
    \node[draw,circle, inner sep=2pt, fill = gray!25] at (v3) [above]{2};
    \node at (v4) [below]{5};
    \node[draw,circle, inner sep=2pt, fill = gray!25] at (v5)  [above]{4};
    \node at (v6) [below]{7};
    \node at (v7)  [above]{6};
    \node at (v8) [below]{\small 2n+3};
    \node at (v9)  [above]{\small 2n+2};
    \node at (v10) [below]{\small 2n+1};
    \node at (v11) [below]{9};
    \node[draw,circle, inner sep=2pt, fill = gray!25] at (v12)  [above]{8};

    \node[draw,circle, inner sep=1pt, fill] at (v13){};
    \node[draw,circle, inner sep=1pt, fill] at (v14){};
    \node[draw,circle, inner sep=1pt, fill] at (v15){};
    
    \node at (v17)  [above]{10};
    \node at (v16) [below]{11};

    \draw [black, thick] (v1) -- (v2);
    \draw  [black, thick]  (v1) -- (v3);
    \draw  [black, thick]  (v3) -- (v2);
     \draw [black, thick] (v2) -- (v4);
    \draw  [black, thick]  (v2) -- (v5);
    \draw  [black, thick]  (v5) -- (v4) ;
     \draw [black, thick] (v4) -- (v6);
    \draw  [black, thick]  (v4) -- (v7);
    \draw  [black, thick]  (v6) -- (v7) ;
    \draw  [black, thick]  (v8) -- (v9);
    \draw  [black, thick]  (v10) -- (v9);
    \draw  [black, thick]  (v8) -- (v10);
    \draw  [black, thick]  (v11) -- (v12);
    \draw  [black, thick]  (v11) -- (v6);
    \draw  [black, thick]  (v12) -- (v6);
    
    \draw [black, thick] (v16) -- (v17);
    \draw [black, thick] (v11) -- (v16);
    \draw [black, thick] (v11) -- (v17);

\end{tikzpicture}
\caption{Position $(\mathcal{G}_n,\textsf{A}_n)$. Vertex $2n+2$ is gray if $n=0$ or $2n+2 \mod 4 = 0$}\label{fig:specialgraphgeneric}
\end{figure}
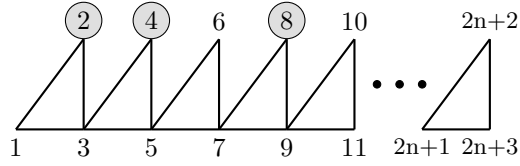

 We seek to solve the decision problem 
\vspace{0.1in}

\begin{center}
\parbox{0.8\textwidth}{
{\flushleft{\bf GRAPH GAME $\textsf{GCDS}(\mathcal{T})$:}}\\
{\flushleft{INSTANCE:}} A position $(\mathcal{F},\textsf{B})$ that is a member of $\mathcal{T}$.
{\flushleft{QUESTION:}} Does ONE have a winning strategy in the game $\textsf{GCDS}(\mathcal{F},\textsf{B})$?\\
  }
\end{center}
\vspace{0.1in}

\begin{theorem}\label{NPGraphs}
Let $(\mathcal{F},\textsf{B})$ be a position which, for some positive integer $n$, is isomorphic to $(\mathcal{G}_n,\textsf{A}_n)$.
\begin{enumerate}
\item{ If $n$ is odd then $(\mathcal{F},\textsf{B})$  is an \emph{\textsf{N}}-position}
\item{If $n$ is even then $(\mathcal{F},\textsf{B})$  is a \emph{\textsf{P}}-position}
\end{enumerate}
\end{theorem}
\begin{proof}
We prove the theorem by induction on $n$.
Observe that for $(\mathcal{G}_1,\textsf{A}_1)$ the graph $\mathcal{G}_1$is a triangle, and one vertex assigned to ONE, and two assigned to TWO.  $(\mathcal{G}_1,\textsf{A}_1)$ is an $N-position$: The next player to move can remove all of the opponent's vertices. 
Note that this is also true in the case of two vertices assigned to ONE and one vertex  assigned to TWO. 

 Now consider $(\mathcal{G}_2,\textsf{A}_2)$, depicted in Figure \ref{fig:G2isP}. As an exhaustive verification shows, application of $\textbf{gcds}_2$ based on any edge of $(\mathcal{G}_2,\textsf{A}_2)$ results in a position isomorphic to an \textsf{N}-position. Thus any position isomorphic to $(\mathcal{G}_2,\textsf{A}_2)$ is a \textsf{P}-position.

\begin{figure}[ht]
\centering
\begin{tikzpicture}[scale = .5]

    \coordinate (v1) at (0,0);
    \coordinate (v2) at (1.5,0);
    \coordinate (v3) at (1.5,2);
    \coordinate (v4) at (3,0);
    \coordinate (v5) at (3,2);
    
    \node at (v1) [below]{1};
    \node at (v2) [below]{3};
    \node[draw,circle, inner sep=2pt, fill = gray!25] at (v3) [above]{2};
    \node at (v4) [below]{5};
    \node[draw,circle, inner sep=2pt, fill = gray!25] at (v5)  [above]{4};
    
    \draw [black, thick] (v1) -- (v2);
    \draw  [black, thick]  (v1) -- (v3);
    \draw  [black, thick]  (v3) -- (v2);
    \draw [black, thick] (v2) -- (v4);
    \draw  [black, thick]  (v2) -- (v5);
    \draw  [black, thick]  (v4) -- (v5);

\end{tikzpicture}
\caption{The position $(\mathcal{G}_2,\textsf{A}_2)$ is a \textsf{P}-position}\label{fig:G2isP}
\end{figure}
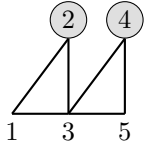

Now let $n\ge 2$ be an integer and assume that for all $k\le n$ it has been verified that if $k$ is even, then $(\mathcal{G}_k,\textsf{A}_k)$ is a \textsf{P}-position, while if $k$ is odd, then $(\mathcal{G}_k,\textsf{A}_k)$ is a \textsf{N}-position.

{\flushleft{Case 1:} $n$ is even.} Then $n+1$ is odd. By Theorem \ref{gcdsonclass} $\textbf{gcds}_2(\mathcal{G}_{n+1},\{2(n+1),\; 2(n+1)+1\} )$ produces from the position $(\mathcal{G}_{n+1},\textsf{A}_{n+1})$ a position isomorphic to $(\mathcal{G}_n,\textsf{A}_n)$, which by the Induction Hypothesis is a \textsf{P}-position. It follows that $(\mathcal{G}_{n+1},\textsf{A}_{n+1})$ is an \textsf{N}-position.

{\flushleft{Case 2:} $n$ is odd.} Then $n+1$ is even, and we must show that $(\mathcal{G}_{n+1},\textsf{A}_{n+1})$ is a \textsf{P}-position. We do this by showing that for each application of $\textbf{gcds}_2$ to $\mathcal{G}_{n+1}$, there is a follow-up  application of $\textbf{gcds}_2$ such that the resulting position is isomorphic to $(\mathcal{G}_{n-1},\textsf{A}_{n-1})$ which, by the induction hypothesis, is a \textsf{P}-position. Thus each application of $\textbf{gcds}_2$ to $\mathcal{G}_{n+1}$ results in a \textsf{N}-position.

This argument involves starting with $(\mathcal{G}_{n+1},\textsf{A}_{n+1})$, considered as the current position after the previous player has moved. We then consider a move by the current player, followed by the previous player again.  We consider two cases:\\
{\flushleft{Case 2.A:}} The current player applies $\textbf{gcds}_2$ to an edge \emph{not} belonging to $\{\{1,\; 2\},\; \{1,\; 3\},\; \{2,\; 3\}\}$. 

In this event we zoom in on a pair of adjacent triangles, called a ``Buddy Block" and marked in dashed edges in Figure \ref{fig:BuddyBlock}, that includes the one that contains the edge targeted by the current player's $\textbf{gcds}_2$ move. 

\begin{figure}[ht]
 \centering
 \begin{tikzpicture}[scale = .7]
    \coordinate (v16) at (-2,1);
    \coordinate (v17) at (-1.5,1);
    \coordinate (v18) at (-1,1);

    \coordinate (v1) at (0,0);
    \coordinate (v2) at (1.5,0);
    \coordinate (v3) at (1.5,2);
    \coordinate (v4) at (3,0);
    \coordinate (v5) at (3,2);
    \coordinate (v6) at (4.5,0);
    \coordinate (v7) at (4.5,2);
    \coordinate (v8) at (9,0);
    \coordinate (v9) at (9,2);
    \coordinate (v10) at (7.5,0);
    \coordinate (v11) at (6,0);
    \coordinate (v12) at (6,2);
    \coordinate (v13) at (6.5,1);
    \coordinate (v14) at (7,1);
    \coordinate (v15) at (7.5,1);

    \node at (v1) [below]{\small 2i+1};
    \node at (v2) [below]{\small 2i+3};
    \node at (v3) [above]{\small 2i+2}; 
    \node at (v4) [below]{\small 2i+5};
    \node at (v5)  [above]{\small 2i+4};
    \node at (v6) [below]{\small 2i+7};
    \node at (v7)  [above]{\small 2i+6};
    \node at (v8) [below]{\small 2n+1};
    \node at (v9)  [above]{\small 2n};
    \node at (v10) [below]{\small 2n-1};
    \node at (v11) [below]{\small 2i+9};
    \node at (v12)  [above]{\small 2i+8};
    \node [draw,circle, inner sep=1pt, fill] at (v13){};
    \node[ draw,circle, inner sep=1pt, fill] at (v14){};
    \node[draw,circle, inner sep=1pt, fill] at (v15){};
    
  \node[draw,circle, inner sep=1pt, fill] at (v16){};
    \node[draw,circle, inner sep=1pt, fill] at (v17){};
    \node[draw,circle, inner sep=1pt, fill] at (v18){};
  
    \draw [black, thick] (v1) -- (v2);
    \draw  [black, thick]  (v1) -- (v3);
    \draw  [black, thick]  (v3) -- (v2);
     \draw [black, thick, densely dotted] (v2) -- (v4);
    \draw  [black, thick, dashed]  (v2) -- (v5);
    \draw  [black, thick, dashed]  (v5) -- (v4) ;
     \draw [black, thick, dashed] (v4) -- (v6);
    \draw  [black, thick, dashed]  (v4) -- (v7);
    \draw  [black, thick, dashed]  (v6) -- (v7) ;
    \draw  [black, thick]  (v8) -- (v9);
    \draw  [black, thick]  (v10) -- (v9);
    \draw  [black, thick]  (v8) -- (v10);
    \draw  [black, thick]  (v11) -- (v12);
    \draw  [black, thick]  (v11) -- (v6);
    \draw  [black, thick]  (v12) -- (v6);
\end{tikzpicture} 
\caption{The configuration in dashes is called a ``Buddy Block"}\label{fig:BuddyBlock}
\end{figure}
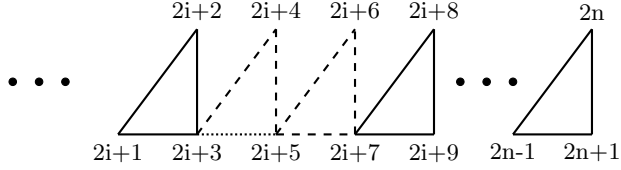

Figures \ref{fig:GenericSegment56}, \ref{fig:GenericSegment57} and \ref{fig:GenericSegment67}  show the three possible outcomes of the current player's move when the move involves edges of the right-hand triangle of the Buddy Block.  Analogous figures arise for moves involving an edge of the left-hand triangle of the Buddy Block. In response to the current player, the previous player also applies a $\textbf{gcds}_2$ move to an edge of the other triangle in the Buddy Block, assuring that one vertex from $\textsf{A}_{n+1}$ and three others are removed.  
See, for example, Figures \ref{fig:Removal1}, \ref{fig:Removal2} and \ref{fig:Removal3}. 

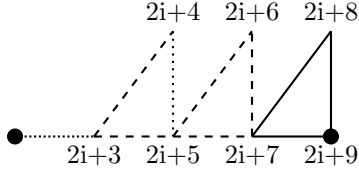
\begin{figure}[h]
\begin{tikzpicture}[scale = .7]

    \coordinate (v0) at (-1,0);
    \coordinate (v1) at (0,0);
    \coordinate (v2) at (1.5,0);
    \coordinate (v3) at (1.5,2);
    \coordinate (v4) at (3,0);
    \coordinate (v5) at (3,2);
    \coordinate (v6) at (4.5,0);
    \coordinate (v7) at (4.5,2);
    \coordinate (v8) at (9,0);
    \coordinate (v9) at (9,2);
    \coordinate (v10) at (6,0);
    \coordinate (v11) at (7.5,0);
    \coordinate (v12) at (7.5,2);
    \coordinate (v13) at (5,1);
    \coordinate (v14) at (5.5,1);
    \coordinate (v15) at (6,0);
       \coordinate (v16) at (6,2);

    \node at (v2) [below]{2i+3};
    \node at (v4) [below]{2i+5};
    \node at (v5)  [above]{2i+4};
    \node at (v6) [below]{2i+7};
    \node at (v7)  [above]{2i+6};
    \node at (v15)  [below]{2i+9};
    \node at (v16)  [above]{2i+8};
    \node[draw,circle, inner sep=2pt, fill] at (v1){};
    \node[draw,circle, inner sep=2pt, fill] at (v15){};
    
    \draw [black, densely dotted, thick] (v1)--(v2);
     \draw [black, thick, dashed] (v2) -- (v4);
    \draw  [black, thick, dashed]  (v2) -- (v5);
    \draw  [black, thick,  dotted]  (v5) -- (v4) ;
     \draw [black, thick, dashed] (v4) -- (v6);
    \draw  [black, thick, dashed]  (v4) -- (v7);
    \draw  [black, thick, dashed]  (v6) -- (v7) ;
    \draw  [black,  thick]  (v15) -- (v6);
  \draw [black, thick] (v6)--(v16);
    \draw [black, thick] (v15)--(v16);
    
\end{tikzpicture}
\caption{Suppose the current player targets the edge $\{2i+4,\; 2i+5\}$.}\label{fig:Removal1}
\end{figure}

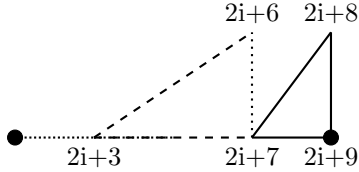
\begin{figure}[ht]
\begin{tikzpicture}[scale = .7]

    \coordinate (v0) at (-1,0);
    \coordinate (v1) at (0,0);
    \coordinate (v2) at (1.5,0);
    \coordinate (v3) at (1.5,2);
    \coordinate (v4) at (3,0);
    \coordinate (v5) at (3,2);
    \coordinate (v6) at (4.5,0);
    \coordinate (v7) at (4.5,2);
    \coordinate (v8) at (9,0);
    \coordinate (v9) at (9,2);
    \coordinate (v10) at (6,0);
    \coordinate (v11) at (7.5,0);
    \coordinate (v12) at (7.5,2);
    \coordinate (v13) at (5,1);
    \coordinate (v14) at (5.5,1);
    \coordinate (v15) at (6,0);
   \coordinate (v16) at (6,2);

    \node at (v2) [below]{2i+3};
    \node at (v6) [below]{2i+7};
    \node at (v7)  [above]{2i+6};
    \node at (v15)  [below]{2i+9};
    \node at (v16)  [above]{2i+8};
    \node[draw,circle, inner sep=2pt, fill] at (v1){};
    \node[draw,circle, inner sep=2pt, fill] at (v15){};
    
     \draw [black, densely dotted, thick] (v1)--(v2);
   \draw [black, densely dotted, thick] (v2)--(v4);
     \draw [black, thick, dashed] (v2) -- (v6);
    \draw  [black, thick, dashed]  (v2) -- (v7);
    \draw  [black, thick, dotted]  (v6) -- (v7) ;
    \draw  [black,  thick]  (v15) -- (v6);
   \draw [black, thick] (v6)--(v16);
    \draw [black, thick] (v15)--(v16);

\end{tikzpicture}
\caption{After $\{2i+4,\; 2i+5\}$ removal, the other player targets the edge $\{2i+6,\; 2i+7\}$.}\label{fig:Removal2}
\end{figure}

\begin{figure}[ht]
\begin{tikzpicture}[scale = .7]
    \coordinate (v0) at (-1,0);
    \coordinate (v1) at (0,0);
    \coordinate (v2) at (1.5,0);
    \coordinate (v3) at (1.5,2);
    \coordinate (v4) at (3,0);
    \coordinate (v5) at (3,2);
    \coordinate (v6) at (4.5,0);
    \coordinate (v7) at (4.5,2);
    \coordinate (v8) at (9,0);
    \coordinate (v9) at (9,2);
    \coordinate (v10) at (6,0);
    \coordinate (v11) at (7.5,0);
    \coordinate (v12) at (7.5,2);
    \coordinate (v13) at (5,1);
    \coordinate (v14) at (5.5,1);
    \coordinate (v15) at (6,0);
   \coordinate (v16) at (6,2);

    \node at (v2) [below]{2i+3};
    \node at (v15) [below]{2i+9};
    \node at (v16) [above]{2i+8};
    \node[draw,circle, inner sep=2pt, fill] at (v15){};
    \node[draw,circle, inner sep=2pt, fill] at (v1){};
    
    \draw [black, dashed, thick] (v2)--(v16);
    \draw [black, dashed, thick] (v2)--(v15);
    \draw [black,  thick] (v15)--(v16);
    \draw [black, densely dotted, thick] (v1)--(v2);
    \draw  [black, dashed, thick]  (v2) -- (v4);
\end{tikzpicture}
\caption{After removal of the edge $\{2i+6,\; 2i+7\}$.}\label{fig:Removal3}
\end{figure}
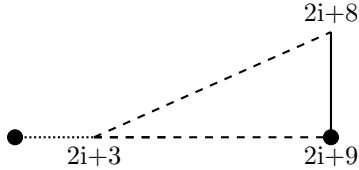

Observe that after these two $\textbf{gcds}_2$ moves the resulting graph is isomorphic to $\mathcal{G}_{n-1}$. Moreover, by considering the cases where $2i+4$ or $2i+6$ is an element of $\textsf{A}_{n+1}$, we see that the resulting position is isomorphic to $(\mathcal{G}_{n-1},\textsf{A}_{n-1})$. A similar analysis shows that for each move by the current player there is a countermove by the previous player so that the result of the two moves is a position isomorphic to $(\mathcal{G}_{n-1},\textsf{A}_{n-1})$. This completes the argument for Case 1.

{\flushleft{Case 2.B:}} The current player applies $\textbf{gcds}_2$ to an edge belonging to $\{\{1,\; 2\},\; \{1,\; 3\},\; \{2,\; 3\}\}$. 

 In this case, we consider the first three triangles and follow the same strategy of removing an edge that results in the loss of one vertex from $\textsf{A}_{n+1}$ and three vertices outside $\textsf{A}_{n+1}$. The resulting position is once again isomorphic to $(\mathcal{G}_{n-1},\textsf{A}_{n-1})$. 

This completes the proof of the theorem.
\end{proof}

Theorem \ref{NPGraphs} provides the following solution to the decision problem GRAPH GAME $\textsf{GCDS}(\mathcal{T})$:
\begin{corollary}\label{specialdecisiongcds} 
For each positive integer $n$ ONE has a winning strategy in the game $(\mathcal{G}_n,\textsf{A}_n)$ if, and only if, $n$ is odd.
\end{corollary}

\section{The permutation sorting game $\textsf{CDS}(\alpha,\textsf{A})$}

The game $\textsf{CDS}(\alpha, A)$ was defined in the introduction, where we also mentioned the following result:

\begin{theorem}[\cite{AHMMSW}, Theorem 4.4]\label{cdsboundtheorem13} Let $\alpha$ be a permutation with a strategic pile $\textsf{P}$. For any set $\textsf{A}\subseteq \textsf{P}$ 
\begin{enumerate}
\item{If $\vert \textsf{A} \vert \ge \frac{3 }{4}\vert \textsf{P} \vert$, then ONE has a winning strategy in $\emph{\textsf{CDS}}(\alpha,\textsf{A})$. }
\item{If $\vert \textsf{P}\vert >0$ and $\vert \textsf{A} \vert \le \max\{\frac{1}{4} \vert\textsf{P} \vert - 2,\; 0\}$, then TWO has a winning strategy in $\emph{\textsf{CDS}}(\alpha,\textsf{A})$. }
\end{enumerate}
\end{theorem}

For the reader's convenience we outline a proof: 

{\flushleft Towards proving (1):}  We may assume that both \textsf{P} and $\textsf{P}\setminus \textsf{A}$ are nonempty, otherwise there is nothing to prove. Consider a strategy of player ONE that calls in each round for removing an element of $\textsf{P}\setminus \textsf{A}$. Consider a play during which ONE followed this strategy. For nonnegative integers $i$ let $P_i$ denote the elements of $\textsf{P}$ not yet removed by the end of round $i$, let $A_i$ denote the set $\textsf{A}\cap P_i$ and let $B_i$ denote $P_i\setminus A_i$. By convention $P_0 = \textsf{P}$, $A_0 = \textsf{A}$ and $B_0 = \textsf{P}\setminus \textsf{A}$.

In a round ONE's strategy may require also removing an element of \textsf{A} to remove an element of $\textsf{P}\setminus \textsf{A}$. Player TWO may be able to remove two elements of \textsf{A} whenever it is TWO's turn. By the \textbf{cds} Bounded Removal Theorem a player can remove at most two elements from the strategic pile in any round. Thus, per round, ONE may loose up to three elements from \textsf{A} while TWO looses only one element from $\textsf{P}\setminus \textsf{A}$. 

By the $i$-th round we have  
$\vert B_i\vert \le \vert \textsf{P}\setminus\textsf{A}\vert - i  \le \frac{1}{4}\vert \textsf{P}\vert - i 
\mbox{ and } 
   \vert A_i\vert \ge \vert A\vert - 3i \ge \frac{3}{4}\vert \textsf{P}\vert - 3i$, so that the inequality $3\vert B_i\vert \le \vert A_i\vert$ holds. By the Strategic Pile Removal Theorem a player may in any round remove any chosen element from the strategic pile, provided there is more than one element in the strategic pile. Thus, for any $i$ such that $\vert B_i\vert > 0$, applying ONE's strategy gives $\vert B_{i+1}\vert <\vert B_i\vert$. Since the game started with a nonempty set $B_0$ there is a last $i\ge 0$ with $B_i$ nonempty. For this $i$ we have $\vert P_i\vert \ge 4 \vert B_i\vert$, so that $P_i$ has at least four elements. In round $i+1$ ONE's strategy removes all elements of $B_i$. By the \textbf{cds} Bounded Removal Theorem, there are elements of $\textsf{A}$ left after this move by ONE. Since $\alpha$ is not \textbf{cds} sortable, ONE wins the play.
{\flushleft Towards proving (2):} 
Note that if $\vert\textsf{A} \vert\le \frac{ \vert \textsf{P} \vert}{4}-2$ then $\vert \textsf{P}\setminus \textsf{A}\vert \ge \frac{3}{4}\vert\textsf{P}\vert + 2$. 
ONE moves first, potentially removing two items belonging to TWO from the strategic pile. After this first move the new strategic pile $\textsf{P}^{\prime}$ has two fewer elements than \textsf{P}, and $\textsf{P}^{\prime}\setminus \textsf{A}$ has two fewer elements than $\textsf{P}\setminus\textsf{A}$. It follows that $\vert \textsf{P}^{\prime}\setminus \textsf{A}\vert \ge \frac{3}{4} \vert\textsf{P}^{\prime}\vert$ and from this point on TWO is the first player. Apply (1) to conclude that TWO has a winning strategy.  

A more detailed analysis of the end-game shows: 
\begin{theorem}\label{endgame} Let $\alpha$ be a permutation with strategic pile \textsf{P}.
\begin{enumerate}
\item{
\begin{enumerate}
  \item{If $\vert \textsf{P}\vert \mod 4 = j$ and $j<2$ then for any $\textsf{A}\subseteq \textsf{P}$ with $\vert\textsf{A}\vert\ge \frac{3}{4}(\vert\textsf{P}\vert-j) +j$, ONE has a winning strategy in $\textsf{CDS}(\alpha,\textsf{A})$.}
  \item{If $\vert \textsf{P}\vert \mod 4 = j$ and $j\ge 2$ then for any $\textsf{A}\subseteq \textsf{P}$ with $\vert\textsf{A}\vert\ge \frac{3}{4}(\vert\textsf{P}\vert-j) + (j-1)$, ONE has a winning strategy in $\textsf{CDS}(\alpha,\textsf{A})$.}
\end{enumerate} 
}
\item{If $\vert \textsf{P}\vert \mod 4 = j$ then for any $\textsf{A}\subseteq\textsf{P}$ with $\vert \textsf{P}\setminus \textsf{A}\vert \ge\frac{3}{4}(\vert \textsf{P}-j) + j$, TWO has a winning strategy in $\textsf{CDS}(\alpha,\textsf{A})$
}
\end{enumerate}
\end{theorem}
In each case the strategy of removing as many strategic pile elements of the opponent per round as possible is a winning strategy of the winning player. The details of the arguments are similar to those outlined for Theorem \ref{cdsboundtheorem13}, and are left to the reader. For $j=3$ Theorem \ref{endgame} (2) can be equivalently stated as follows:
\begin{quote}
{\tt If $\vert\textsf{P}\vert \mod 4 = 3$ then for all $\textsf{A}\subseteq \textsf{P}$ such that $\vert \textsf{A}\vert \le \frac{1}{4}(\vert\textsf{P}\vert-3)$, TWO has a winning strategy in the game $\textsf{CDS}(\alpha,\textsf{A})$.}
\end{quote}
In the next theorem we prove that  this bound on $\vert\textsf{A}\vert$ is optimal:
\begin{theorem}\label{boundtighttheorem} For each integer $n>4$ that is a multiple of $4$ there exists a permutation $\alpha_n\in\textsf{S}_n$ with a strategic pile $\textsf{P}_n$, and a set $\textsf{A}_n\subseteq \textsf{P}_n$ such that
\begin{enumerate}
\item{$\vert \textsf{P}_n\vert \mod 4 = 3$,}
\item{$\vert \textsf{A}_n\vert =\frac{\vert\textsf{P}_n\vert - 3}{4} + 1$, and}
\item{ONE has a winning strategy in the game $\textsf{CDS}(\alpha_n,\textsf{A}_n)$.}
\end{enumerate}
\end{theorem}
\begin{proof}

Consider $\alpha_n = 
\lbrack{\color{black}5,7,6}, \cdots \underbrace{k,k-1,k+2,k+1}, \cdots ,\; \underbrace{n-3,\; n-4,\; n-1,\;n-2},\;{\color{black}3,2,4},\;n,\;{\color{black}1}\rbrack$ for $n>4$ a multiple of $4$. 
For odd $k>1$ consider the pointers in $\alpha_n$ associated with the repeating pattern $[k,k-1,k+2,k+1]$:
\begin{itemize}
\item The $(k-1,k)$ pointer interlocks with both the $(k,k+1)$ and the $(k-2,k-1)$ pointer. 
\item The $(k,k+1)$ pointer interlocks with $(k-2,k-1), (k-1,k), (k,+1,k+2)$ and $(k+2,k+3)$ pointers. 
\item The $(k+1,k+2)$ pointer interlocks with the $(k+2, k+3)$ and $(k, k+1)$ pointer. 
\end{itemize}
Thus, we get get the partial overlap graph in Figure \ref{fig:Return}.  

\begin{figure}[h]
\centering
\begin{tikzpicture}[scale = 1]

    \coordinate (v2) at (0,0);
    \coordinate (v4) at (2,0);
    \coordinate (v5) at (2,2);
    \coordinate (v6) at (4,0);
    \coordinate (v7) at (4,2);

    \node at (v2) [below]{(k-2,k-1)};
    \node at (v4) [below]{(k,k+1)};
    \node at (v5)  [above]{(k-1, k)};
    \node at (v6) [below]{(k+1,k+2)};
    \node at (v7)  [above]{(k+2,k+3)};

     \draw [black, thick] (v2) -- (v4);
    \draw  [black, thick]  (v2) -- (v5);
    \draw  [black, thick]  (v5) -- (v4) ;
     \draw [black, thick] (v4) -- (v6);
    \draw  [black, thick]  (v4) -- (v7);
    \draw  [black, thick]  (v6) -- (v7) ;
    
\end{tikzpicture}
\caption{The Buddy Block returns!}\label{fig:Return}
\end{figure}
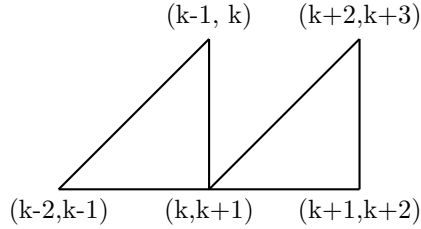

The full overlap graph $\mathcal{O}_n$ of $\alpha_n$ is isomorphic to the graph $\mathcal{G}_{n-1}$ and is depicted in Figure \ref{fig:General}. Let $f:\mathcal{G}_{n-1}\rightarrow\mathcal{O}_n$ be the isomorphism. 

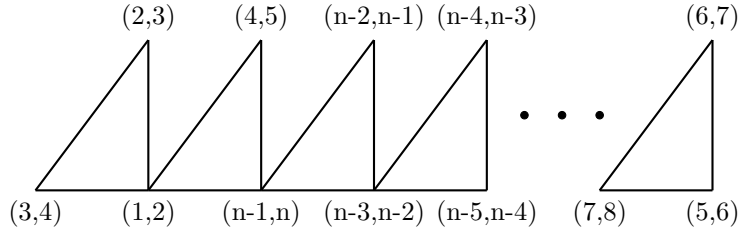
\begin{figure}[ht]
\centering
\begin{tikzpicture}[scale = 1]

    \coordinate (v1) at (0,0);
    \coordinate (v2) at (1.5,0);
    \coordinate (v3) at (1.5,2);
    \coordinate (v4) at (3,0);
    \coordinate (v5) at (3,2);
    \coordinate (v6) at (4.5,0);
    \coordinate (v7) at (4.5,2);
    \coordinate (v8) at (9,0);
    \coordinate (v9) at (9,2);
    \coordinate (v10) at (7.5,0);
    \coordinate (v11) at (6,0);
    \coordinate (v12) at (6,2);
    \coordinate (v13) at (6.5,1);
    \coordinate (v14) at (7,1);
    \coordinate (v15) at (7.5,1);

    \node at (v1) [below]{(3,4)};
    \node at (v2) [below]{(1,2)};
    \node at (v3) [above]{(2,3)};
    \node at (v4) [below]{(n-1,n)};
    \node at (v5)  [above]{(4,5)};
    \node at (v6) [below]{(n-3,n-2)};
    \node at (v7)  [above]{(n-2,n-1)};
    \node at (v8) [below]{(5,6)};
    \node at (v9)  [above]{(6,7)};
    \node at (v10) [below]{(7,8)};
    \node at (v11) [below]{(n-5,n-4)};
    \node at (v12)  [above]{(n-4,n-3)};
    \node[draw,circle, inner sep=1pt, fill] at (v13){};
    \node[draw,circle, inner sep=1pt, fill] at (v14){};
    \node[draw,circle, inner sep=1pt, fill] at (v15){};

    \draw [black, thick] (v1) -- (v2);
    \draw  [black, thick]  (v1) -- (v3);
    \draw  [black, thick]  (v3) -- (v2);
     \draw [black, thick] (v2) -- (v4);
    \draw  [black, thick]  (v2) -- (v5);
    \draw  [black, thick]  (v5) -- (v4) ;
     \draw [black, thick] (v4) -- (v6);
    \draw  [black, thick]  (v4) -- (v7);
    \draw  [black, thick]  (v6) -- (v7) ;
    \draw  [black, thick]  (v8) -- (v9);
    \draw  [black, thick]  (v10) -- (v9);
    \draw  [black, thick]  (v8) -- (v10);
    \draw  [black, thick]  (v11) -- (v12);
    \draw  [black, thick]  (v11) -- (v6);
    \draw  [black, thick]  (v12) -- (v6);
\end{tikzpicture}
\caption{General form of the move graph of $\alpha_n$.}\label{fig:General}
\end{figure}

Next observe that these permutations have a full strategic pile, so that every vertex in the move graph represents an element in the strategic pile. From the definition of $\alpha_n$ we see that the segment 
\[
   {\color{black}\DashedArrow[ dotted, thick    ]} 5 {\color{black!80}\rightarrow} 0 {\color{black}\DashedArrow[ dotted, thick    ]} 1 {\color{black!80}\rightarrow} n {\color{black}\DashedArrow[ dotted, thick    ]} n+1 {\color{black!80}\rightarrow} 1 {\color{black}\DashedArrow[ dotted, thick    ]} 2 {\color{black!}\rightarrow} 3 {\color{black}\DashedArrow[ dotted, thick    ]} 4 {\color{black!80}\rightarrow} 2{\color{black}\DashedArrow[ dotted, thick    ]} 3\cdots
\]
is part of an alternating path in the cycle graph of $\alpha_n$. Thus the strategic pile is nonempty. We must show that for each $n>4$ that is a multiple of $4$, for each $k$ with $1\le k<n$ the edge $k{\color{black}\DashedArrow[ dotted, thick    ]} k+1$ is in this alternating cycle. For $n>8$, continuing the alternating path from the entry $3$ in $\alpha_n$ shows that each of the edges $(n-2i){\color{black}\DashedArrow[ dotted, thick    ]} (n-2i+1)$ for $1\le i \le (n/2-4)$ occurs in this alternating cycle. Since $\alpha_n$ is of form $\lbrack 5,\; 7,\; 6,\; 9,\; 8,\; 11,\; 10,\; \cdots\rbrack$ we see that this alternating cycle also contains 
\[
   \cdots {\color{black!80}\rightarrow} 6 {\color{black}\DashedArrow[ dotted, thick    ]} 7 {\color{black!80}\rightarrow} 5 {\color{black}\DashedArrow[ dotted, thick    ]} 6 {\color{black!80}\rightarrow} 7 \cdots.
\]
It follows, by continuing the alternating cycle following permutation $\alpha_n$'s entries, that each edge of form $n- 2i -1{\color{black}\DashedArrow[ dotted, thick    ]} n-2i$ for $1\le i < (n/2-4)$ is also in this alternating cycle. Examining the final segment of $\alpha_n$ we see that this cycle continues as
\[
   \cdots (n-2){\color{black!80}\rightarrow} (n-1) {\color{black}\DashedArrow[ dotted, thick    ]} n {\color{black!80}\rightarrow} 4 {\color{black}\DashedArrow[ dotted, thick    ]} 5 {\color{black!80}\rightarrow} 0 \cdots
\]
showing that the remaining edges $i{\color{black}\DashedArrow[ dotted, thick    ]}i+1$ are also in this alternating cycle. It follows that $\textsf{P}_n = \{(i,i+1):\; 1\le i<n\}$ is the strategic pile of $\alpha_n$. Thus $\textsf{P}_n$ is the vertex set of overlap the graph $\mathcal{O}_n$ of $\alpha_n$. With $f:\mathcal{G}_n\rightarrow\mathcal{O}_n$ the graph isomorphism, choose the subset $\textsf{B}_n = \{f(a):a\in\textsf{A}_n\}$ of $\textsf{P}_n$. Then $(\mathcal{O}_n,\textsf{B}_n)$ is isomorphic to the position $(\mathcal{G}_{n-1},\textsf{A}_{n-1})$.

As $n-1$ is odd, Corollary \ref{specialdecisiongcds} implies that ONE has a winning strategy in the game $\textsf{GCDS}(\mathcal{G}_{n-1}, \textsf{A}_{n-1})$, and thus in the game $\textsf{GCDS}(\mathcal{O}_n,\textsf{B}_n)$. By its definition $\vert \textsf{A}_{n-1}\vert =  \frac{\vert\textsf{P}_n\vert-3}{4} +1$. By Theorem \ref{movegraphcds} ONE has a winning strategy in the game $\textsf{CDS}(\alpha_n,\textsf{B}_n)$. 
\end{proof}
Note that by the end-game analysis given before Theorem \ref{boundtighttheorem}, removing one point from $\textsf{B}_n$ creates an example where TWO has a winning strategy in the corresponding game. This illustrates that the bound in item (2) of the end-game analysis for player TWO is optimal.

\section{Remarks and Questions}

Can Theorem \ref{boundtighttheorem} be further improved? Theorem \ref{boundtighttheorem} gives examples in $\textsf{S}_n$ when $n$ is a multiple of $4$. This constraint on $n$ seems to be an artifact of the proof technique, but at this time there is no indication that for other classes of $n$ the corresponding tight bound might be achievable. 
One approach to this question may be to examine the corresponding problem for graphs that are the move graphs of permutations. This raises another interesting decision problem:
\vspace{0.1in}

\begin{center}
\parbox{0.8\textwidth}{
{\flushleft{\bf CDS MOVE GRAPH}}\\
{\flushleft{INSTANCE:}} A finite graph $\mathcal{F}$.
{\flushleft{QUESTION:}} Is there a permutation $\alpha$ with \textbf{cds} move graph isomorphic to $\mathcal{F}$?\\
  }
\end{center}
\vspace{0.1in}
It is of independent interest to determine the complexity of the CDS MOVE GRAPH decision problem.

The fundamental facts exploited in Theorem \ref{boundtighttheorem} are (1) that there are connected finite graphs $\mathcal{G} = (V,\; E)$ and a set $\textsf{A}\subseteq V$ such that $\vert V\vert \mod 4 = 3$ and  $\vert\textsf{A}\vert = \frac{\vert V\vert-3}{4} + 1$ and yet ONE has a winning strategy in the game $\textsf{GCDS}(\mathcal{G},\textsf{A})$ and (2) these graphs are isomorphic to the \textbf{cds} move graphs of permutations with a full strategic pile. It is not clear that for connected graphs without an additional constraint such as (2), there should be a threshold proportion such that for $\frac{\vert \textsf{A}\vert}{\vert V\vert}$ above this proportion, ONE has a winning strategy while for $\frac{\vert \textsf{A}\vert}{\vert V\vert}$ below this proportion TWO has a winning strategy\footnote{Thus it is not clear if a generalization of Theorem \ref{cdsboundtheorem13} holds for finite connected graphs.}.

\begin{problem}\label{graphoptimality}
Is there for each $\epsilon>0$ a connected finite graph $\mathcal{G} = (V,\; E)$, and set $\textsf{A}\subseteq V$ is such that $\frac{\vert \textsf{A}\vert}{\vert V\vert}\le \epsilon$, and yet ONE has a winning strategy in the game $\textsf{GCDS}(\mathcal{G},\textsf{A})$?
\end{problem}

\section{Acknowledgements}

The research represented in this paper was funded by an NSF REU grant DMS 1359425, by Boise State University and by the Department of Mathematics at Boise State University. We thank the authors of \cite{AHMMSW} for pre-publication access to their results.

\vspace{0.05in}

{\flushleft{* Corresponding Author: }} mscheepe@boisestate.edu
\vspace{-0.1in}


\begin{thebibliography}{}

\bibitem{AHMMSW} K. Adamyk, E. Holmes, G. Mayfield, D.J. Moritz, M. Scheepers, B.E. Tenner and H. Wauck, \emph{Sorting permutations: Games, genomes and graphs}, arXiv:1410.2353

\bibitem{BP} V. Bafna and P.A. Pevzner, \emph{Sorting by transpositions}, in {\bf Proceedings of the 6$^{th}$ annual ACM-SIAM Symposium on Discrete Algorithms} (1995), 614 - 623.

\bibitem{DC} D.A. Christie, \emph{Sorting permutations by block-interchanges}, {\bf Information Processing Letters} 60 (1996), 165 - 169.

\bibitem{GJ} M.R. Garey and D.S. Johnson, \emph{Computers and Intractability: A guide to the theory of NP completeness}, {\bf Freeman}, 1979.

\bibitem{HPR} S. Hannenhalli and P.A. Pevzner, \emph{Transforming cabbage into turnip: Polynomial algorithm for sorting signed permutations by reversals}, {\bf Journal of the ACM} 46:1 (1999), 1 - 27.

\bibitem{PER} D.M. Prescott, A. Ehrenfeucht and G. Rozenberg, \emph{Molecular operations for DNA processing in hypotrichous ciliates}, {\bf European Journal of Protistology} 37 (2001), 241–260.

\bibitem{TBS} E. Tannier, A. Bergeron and M-F. Sagot, \emph{Advances on sorting by reversals}, {\bf Discrete Applied Mathematics} 155 (2007), 881 - 888.

\bibitem{Z} E. Zermelo, \emph{\"{U}ber eine Anwendung der Mengenlehre auf die Theorie des Schachspiels}, {\bf Proceedings of the Fifth Congress of Mathematicians, Cambridge University Press} (1913), 501 - 504.

\end{thebibliography}
\end{document}